\documentclass[12pt,a4paper,twoside]{article}

\title{\sc Low Frequency Asymptotics for Time-Harmonic Generalized Maxwell Equations in Nonsmooth Exterior Domains}
\def\shorttitle{Low Frequency Asymptotics for Maxwell's Equations}
\def\pauthor{Dirk Pauly}

\def\mylabelonoff{off}
\def\allowdisbrk{no}

\usepackage{a4,exscale,ifthen,amsfonts,amssymb,amsmath,amscd,graphicx}
\usepackage[english]{babel}
\usepackage[mathscr]{eucal}

\author{{\sf\pauthor}}
\numberwithin{equation}{section}

\newenvironment{acknow}{{\vspace*{1cm}\noindent\bf Acknowledgements }}{}
\newcommand{\bewboxw}{\mbox{}\hfill $\square$ \\}

\newenvironment{proof}{{\noindent\bf Proof }}{\bewboxw}

\newcommand{\keywords}[1]{{\noindent\bf Key Words }#1}
\newcommand{\amsclass}[1]{{\noindent\bf AMS MSC-Classifications }#1}

\ifthenelse{\equal{\mylabelonoff}{on}}
{\newcommand{\mylabel}[1]{\label{#1}\fbox{{\rm #1}}}}{\newcommand{\mylabel}[1]{\label{#1}\makebox[0mm][]{}}}
\ifthenelse{\equal{\allowdisbrk}{yes}}
{\allowdisplaybreaks}{}

\newcommand{\paper}[7]{\bibitem{#1} #2, `#3', {\it #4}, #5, (#6), #7.}

\newcommand{\book}[6]{\bibitem{#1} #2, {\it #3}, #4, #5, (#6).}

\newcommand{\dissavail}[6]{\bibitem{#1} #2, `#3', {\sf Dissertation}, #4, (#5), available from {\tt #6}.}
\newcommand{\habil}[5]{\bibitem{#1} #2, `#3', {\sf Habilitationsschrift}, #4, (#5).}

\usepackage{amsfonts,amssymb,amsmath,amscd}
\usepackage[mathscr]{eucal}

\newcommand{\schluss}{\ifodd\value{page}\newpage\thispagestyle{empty}\makebox[0mm][]{}\color{sehrhell}.\fi\end{document}}
\newcommand{\non}{\nonumber}
\newcommand{\normabst}{\hspace{-0.4ex}}

\newcommand{\ol}[1]{\overline{#1}}

\newcommand{\ub}{\underbrace}
\newcommand{\qtext}[1]{\quad\text{#1}\quad}
\newcommand{\qqtext}[1]{\qquad\text{#1}\qquad}
\newcommand{\peh}{\frac{1}{2}}
\newcommand{\cast}{\circledast}
\newcommand{\Nh}{\frac{N}{2}}
\newcommand{\meh}{{-\peh}}
\newcommand{\me}{{-1}}

\newcommand{\intom}{\int_\Omega}

\newcommand{\om}{{\Omega}}

\newcommand{\omq}{\ol{\om}}

\newcommand{\dom}{{\p\om}}

\newcommand{\ohne}{\setminus}

\newcommand{\eps}{\varepsilon}
\newcommand{\epsn}{\eps_0}
\newcommand{\epsd}{\hat{\eps}}

\newcommand{\mun}{\mu_0}
\newcommand{\mud}{\hat{\mu}}

\newcommand{\Lambdad}{\hat{\Lambda}}

\newcommand{\EH}{(E,H)}
\newcommand{\EHn}{(E_n,H_n)}
\newcommand{\EHm}{(E_m,H_m)}
\newcommand{\EHs}{(\tilde{E},\tilde{H})}

\newcommand{\eh}{(e,h)}

\newcommand{\FG}{(F,G)}

\newcommand{\FGs}{(\tilde{F},\tilde{G})}
\newcommand{\FGn}{(F_n,G_n)}
\newcommand{\FGm}{(F_m,G_m)}

\newcommand{\fg}{(f,g)}
\newcommand{\bon}[1]{\overset{\circ}{b}{}^{#1}}

\newcommand{\FdGr}{(F^d,G^r)}
\newcommand{\FrGdn}{(F^r_n,G^d_n)}
\newcommand{\FdGrn}{(F^d_n,G^r_n)}
\newcommand{\ErHd}{(E^r,H^d)}

\newcommand{\To}{\longrightarrow}

\newcommand{\Equi}{\Longleftrightarrow}

\newcommand{\restr}[2]{\left.#1 \right|_{#2}}
\newcommand{\map}[3]{#1\;:\;#2\;\To\;#3}
\newcommand{\Abb}[5]{\begin{array}{ccccc}#1&:&#2&\longrightarrow&#3\\{}&{}&#4&\longmapsto&#5\end{array}}

\newcommand{\bds}{\begin{displaystyle}}
\newcommand{\eds}{\end{displaystyle}}
\newcommand{\beq}{\begin{equation}}
\newcommand{\eeq}{\end{equation}}
\newcommand{\beqa}{\begin{eqnarray}}
\newcommand{\eeqa}{\end{eqnarray}}
\newcommand{\beqanon}{\begin{eqnarray*}}
\newcommand{\eeqanon}{\end{eqnarray*}}
\newcommand{\bamath}{$$\begin{array}}
\newcommand{\eamath}{\end{array}$$}
\newcommand{\ba}{\begin{array}}
\newcommand{\ea}{\end{array}}
\newcommand{\bc}{\begin{center}}
\newcommand{\ec}{\end{center}}
\newcommand{\bmini}{\begin{minipage}}
\newcommand{\emini}{\end{minipage}}

\newtheorem{lem}{Lemma}[section]
\newtheorem{theo}[lem]{Theorem}
\newtheorem{cor}[lem]{Corollary}
\newtheorem{defini}[lem]{Definition}
\newtheorem{rem}[lem]{Remark}

\newcommand{\set}[2]{\{#1 \;\text{\bf :}\;  #2\}}

\newcommand{\setb}[2]{\big\{#1 \;\text{\bf :}\;  #2\big\}}

\DeclareMathOperator{\B}{B}

\newcommand{\pb}[2]{\overset{#2}{\B}{}^{#1}}

\newcommand{\bqom}{\B^q(\Omega)}
\newcommand{\bonqom}{\pb{q}{\circ}(\Omega)}
\newcommand{\bqpeom}{\B^{q+1}(\Omega)}

\newcommand{\rz}{\mathbb{R}}
\newcommand{\rzon}{\rz\ohne\{0\}}
\newcommand{\nz}{\mathbb{N}}
\newcommand{\zz}{\mathbb{Z}}

\newcommand{\cz}{\mathbb{C}}
\newcommand{\czp}{{\cz_+}}

\newcommand{\czpomd}{\cz_{+,\hat{\omega}}}
\newcommand{\nzn}{{\nz_0}}
\newcommand{\rd}{{\rz^3}}
\newcommand{\rN}{{\rz^N}}

\newcommand{\rzp}{{\rz_+}}

\newcommand{\pP}{\mathbb{P}}

\newcommand{\calF}{{\mathscr{F}}}

\newcommand{\calO}{{\mathscr{O}}}

\newcommand{\calH}{{\mathscr{H}}}

\DeclareMathOperator{\cR}{{\mathcal{R}}}

\newcommand{\calM}{{\mathscr{M}}}

\newcommand{\calL}{{\mathcal{L}}}
\newcommand{\calN}{{\mathscr{N}}}

\newcommand{\cS}{{\mathcal{S}}}

\newcommand{\zvec}[2]{\begin{bmatrix}#1\\#2\end{bmatrix}}

\newcommand{\zmat}[4]{\begin{bmatrix}#1&#2\\#3&#4\end{bmatrix}}

\DeclareMathOperator{\p}{\partial}

\DeclareMathOperator{\loes}{\calL}
\DeclareMathOperator{\loesom}{\calL_\omega}
\DeclareMathOperator{\loesn}{\calL_0}
\DeclareMathOperator{\loc}{loc}
\DeclareMathOperator{\vox}{vox}
\DeclareMathOperator{\Max}{\mathsf{Max}}

\DeclareMathOperator{\Lin}{Lin}
\DeclareMathOperator{\id}{Id}

\DeclareMathOperator{\imt}{Im}
\DeclareMathOperator{\supp}{supp}
\DeclareMathOperator{\dist}{dist}
\DeclareMathOperator{\rot}{rot}
\DeclareMathOperator{\Rot}{Rot}

\DeclareMathOperator{\pdiv}{div}

\DeclareMathOperator{\ie}{i}

\DeclareMathOperator{\pd}{d\hspace{-0.4ex}}

\DeclareMathOperator{\curl}{curl}

\newcommand{\pcoverset}[2]{\overset{#2}{\mathrm{C}}{}^{#1}}
\newcommand{\pc}[1]{\pcoverset{#1}{}}

\newcommand{\cu}{\pc{\infty}}

\newcommand{\lz}{\mathrm{L}^2}
\newcommand{\lzom}{\lz(\Omega)}
\newcommand{\Lz}[1]{\lz_{#1}}

\newcommand{\Lzs}{\Lz{s}}
\newcommand{\Lzsom}{\Lz{s}(\Omega)}
\newcommand{\Lzt}{\Lz{t}}

\newcommand{\Lzloc}{\Lz{\loc}}
\newcommand{\Lzlocom}{\Lzloc(\Omega)}

\newcommand{\h}[3]{\overset{#3}{\mathrm{H}}{}^{#1}_{#2}}
\renewcommand{\H}[3]{\overset{#3}{\mathbf{H}}{}^{#1}_{#2}}

\newcommand{\hm}{\h{m}{}{}}
\newcommand{\Hm}{\H{m}{}{}}
\newcommand{\hms}{\h{m}{s}{}}
\newcommand{\Hms}{\H{m}{s}{}}

\renewcommand{\hom}[2]{\h{#1}{#2}{}(\Omega)}
\newcommand{\Hom}[2]{\H{#1}{#2}{}(\Omega)}

\newcommand{\hmom}{\hm(\Omega)}
\newcommand{\Hmom}{\Hm(\Omega)}
\newcommand{\hmsom}{\hms(\Omega)}
\newcommand{\Hmsom}{\Hms(\Omega)}

\newcommand{\qc}[3]{\overset{#3}{\mathrm{C}}{}^{#1,#2}}
\newcommand{\cq}[1]{\qc{#1}{q}{}}
\newcommand{\cqn}[1]{\qc{#1}{q}{\circ}}
\newcommand{\cqu}{\cq{\infty}}
\newcommand{\cqun}{\cqn{\infty}}

\newcommand{\cqunom}{\cqun(\Omega)}

\newcommand{\cqpe}[1]{\qc{#1}{q+1}{}}
\newcommand{\cqpen}[1]{\qc{#1}{q+1}{\circ}}
\newcommand{\cqpeu}{\cqpe{\infty}}
\newcommand{\cqpeun}{\cqpen{\infty}}

\newcommand{\qlz}[1]{\mathrm{L}^{2,#1}}
\newcommand{\qlzom}[1]{\qlz{#1}(\Omega)}
\newcommand{\lzq}{\qlz{q}}
\newcommand{\lzqom}{\qlzom{q}}
\newcommand{\qLz}[2]{\qlz{#1}_{#2}}
\newcommand{\qLzom}[2]{\qLz{#1}{#2}(\Omega)}
\newcommand{\Lzq}[1]{\qLz{q}{#1}}
\newcommand{\Lzqom}[1]{\Lzq{#1}(\Omega)}
\newcommand{\qLzs}[1]{\qLz{#1}{s}}

\newcommand{\qLzt}[1]{\qLz{#1}{t}}

\newcommand{\Lzqs}{\qLzs{q}}
\newcommand{\Lzqsom}{\Lzqs(\Omega)}
\newcommand{\Lzqt}{\qLzt{q}}
\newcommand{\Lzqtom}{\Lzqt(\Omega)}

\newcommand{\Lzqloc}{\Lzq{\loc}}
\newcommand{\Lzqlocom}{\Lzqloc(\Omega)}

\newcommand{\lzqpe}{\qlz{q+1}}
\newcommand{\lzqpeom}{\qlzom{q+1}}
\newcommand{\Lzqpe}[1]{\qLz{q+1}{#1}}
\newcommand{\Lzqpeom}[1]{\Lzqpe{#1}(\Omega)}
\newcommand{\Lzqpes}{\qLzs{q+1}}
\newcommand{\Lzqpesom}{\Lzqpes(\Omega)}
\newcommand{\Lzqpet}{\qLzt{q+1}}
\newcommand{\Lzqpetom}{\Lzqpet(\Omega)}
\newcommand{\Lzqpeloc}{\Lzqpe{\loc}}
\newcommand{\Lzqpelocom}{\Lzqpeloc(\Omega)}

\newcommand{\Lzqgmeh}{\Lzq{>\meh}}

\newcommand{\Lzqgeh}{\Lzq{>\peh}}
\newcommand{\Lzqgehom}{\Lzqom{>\peh}}

\newcommand{\Lzqpegmeh}{\Lzqpe{>\meh}}

\newcommand{\Lzqpegeh}{\Lzqpe{>\peh}}
\newcommand{\Lzqpegehom}{\Lzqpeom{>\peh}}

\newcommand{\qh}[4]{\overset{#4}{\mathrm{H}}{}^{#1,#2}_{#3}}
\newcommand{\qH}[4]{\overset{#4}{\mathbf{H}}{}^{#1,#2}_{#3}}

\newcommand{\qhom}[4]{\qh{#1}{#2}{#3}{#4}(\Omega)}
\newcommand{\qHom}[4]{\qH{#1}{#2}{#3}{#4}(\Omega)}

\newcommand{\hqmsom}{\qhom{m}{q}{s}{}}
\newcommand{\Hqmsom}{\qHom{m}{q}{s}{}}

\newcommand{\pr}[4]{{}_{#3}\overset{#4}{\mathrm{R}}{}^{#1}_{#2}}
\newcommand{\prq}[1]{\pr{q}{#1}{}{}}
\newcommand{\rqpe}[1]{\pr{q+1}{#1}{}{}}
\newcommand{\rqs}{\pr{q}{s}{}{}}

\newcommand{\rqloc}{\pr{q}{\loc}{}{}}

\newcommand{\rqvox}{\pr{q}{\vox}{}{}}

\newcommand{\rqpen}[1]{\pr{q+1}{#1}{0}{}}
\newcommand{\rqsn}{\pr{q}{s}{0}{}}

\newcommand{\ronq}[1]{\pr{q}{#1}{}{\circ}}
\newcommand{\ronqpe}[1]{\pr{q+1}{#1}{}{\circ}}
\newcommand{\ronqs}{\pr{q}{s}{}{\circ}}

\newcommand{\ronqt}{\pr{q}{t}{}{\circ}}

\newcommand{\ronqloc}{\pr{q}{\loc}{}{\circ}}
\newcommand{\ronqpeloc}{\pr{q+1}{\loc}{}{\circ}}

\newcommand{\ronqn}[1]{\pr{q}{#1}{0}{\circ}}

\newcommand{\ronqsn}{\pr{q}{s}{0}{\circ}}

\newcommand{\rqom}[1]{\prq{#1}(\Omega)}
\newcommand{\rqpeom}[1]{\rqpe{#1}(\Omega)}
\newcommand{\rqsom}{\rqs(\Omega)}

\newcommand{\rqlocomq}{\rqloc(\ol{\Omega})}

\newcommand{\rqvoxom}{\rqvox(\Omega)}

\newcommand{\rqpenom}[1]{\rqpen{#1}(\Omega)}

\newcommand{\ronqom}[1]{\ronq{#1}(\Omega)}
\newcommand{\ronqpeom}[1]{\ronqpe{#1}(\Omega)}
\newcommand{\ronqsom}{\ronqs(\Omega)}

\newcommand{\ronqtom}{\ronqt(\Omega)}

\newcommand{\ronqlocom}{\ronqloc(\Omega)}
\newcommand{\ronqpelocom}{\ronqpeloc(\Omega)}
\newcommand{\ronqlocomq}{\ronqloc(\ol{\Omega})}

\newcommand{\ronqnom}[1]{\ronqn{#1}(\Omega)}

\newcommand{\ronqsnom}{\ronqsn(\Omega)}

\newcommand{\pR}[4]{{}_{#3}\overset{#4}{\mathbf{R}}{}^{#1}_{#2}}
\newcommand{\pRq}[1]{\pR{q}{#1}{}{}}
\newcommand{\Rqpe}[1]{\pR{q+1}{#1}{}{}}
\newcommand{\Rqs}{\pR{q}{s}{}{}}
\newcommand{\Rqpes}{\pR{q+1}{s}{}{}}
\newcommand{\Rqt}{\pR{q}{t}{}{}}

\newcommand{\Ronq}[1]{\pR{q}{#1}{}{\circ}}
\newcommand{\Ronqpe}[1]{\pR{q+1}{#1}{}{\circ}}
\newcommand{\Ronqs}{\pR{q}{s}{}{\circ}}
\newcommand{\Ronqpes}{\pR{q+1}{s}{}{\circ}}
\newcommand{\Ronqt}{\pR{q}{t}{}{\circ}}
\newcommand{\Ronqpet}{\pR{q+1}{t}{}{\circ}}

\newcommand{\Ronqpetn}{\pR{q+1}{t}{0}{\circ}}
\newcommand{\Rqom}[1]{\pRq{#1}(\Omega)}
\newcommand{\Rqpeom}[1]{\Rqpe{#1}(\Omega)}
\newcommand{\Rqsom}{\Rqs(\Omega)}
\newcommand{\Rqpesom}{\Rqpes(\Omega)}
\newcommand{\Rqtom}{\Rqt(\Omega)}

\newcommand{\Ronqom}[1]{\Ronq{#1}(\Omega)}
\newcommand{\Ronqpeom}[1]{\Ronqpe{#1}(\Omega)}
\newcommand{\Ronqsom}{\Ronqs(\Omega)}
\newcommand{\Ronqpesom}{\Ronqpes(\Omega)}
\newcommand{\Ronqtom}{\Ronqt(\Omega)}
\newcommand{\Ronqpetom}{\Ronqpet(\Omega)}

\newcommand{\Ronqpetnom}{\Ronqpetn(\Omega)}
\newcommand{\Ronqkmeh}{\Ronq{<\meh}}
\newcommand{\Ronqkmehom}{\Ronqkmeh(\Omega)}

\newcommand{\Rqkmeh}{\pRq{<\meh}}
\newcommand{\Rqkmehom}{\Rqkmeh(\Omega)}

\newcommand{\pdi}[4]{{}_{#3}\overset{#4}{\mathrm{D}}{}^{#1}_{#2}}
\newcommand{\pdq}[1]{\pdi{q}{#1}{}{}}
\newcommand{\dqpe}[1]{\pdi{q+1}{#1}{}{}}
\newcommand{\dqs}{\pdi{q}{s}{}{}}
\newcommand{\dqpes}{\pdi{q+1}{s}{}{}}

\newcommand{\dqloc}{\pdi{q}{\loc}{}{}}
\newcommand{\dqpeloc}{\pdi{q+1}{\loc}{}{}}
\newcommand{\dqvox}{\pdi{q}{\vox}{}{}}

\newcommand{\dqn}[1]{\pdi{q}{#1}{0}{}}

\newcommand{\dqsn}{\pdi{q}{s}{0}{}}

\newcommand{\dqom}[1]{\pdq{#1}(\Omega)}
\newcommand{\dqpeom}[1]{\dqpe{#1}(\Omega)}
\newcommand{\dqsom}{\dqs(\Omega)}
\newcommand{\dqpesom}{\dqpes(\Omega)}

\newcommand{\dqlocom}{\dqloc(\Omega)}
\newcommand{\dqpelocom}{\dqpeloc(\Omega)}

\newcommand{\dqvoxom}{\dqvox(\Omega)}

\newcommand{\dqnom}[1]{\dqn{#1}(\Omega)}

\newcommand{\dqsnom}{\dqsn(\Omega)}

\newcommand{\pDi}[4]{{}_{#3}\overset{#4}{\mathbf{D}}{}^{#1}_{#2}}
\newcommand{\pDq}[1]{\pDi{q}{#1}{}{}}
\newcommand{\Dqpe}[1]{\pDi{q+1}{#1}{}{}}
\newcommand{\Dqs}{\pDi{q}{s}{}{}}
\newcommand{\Dqpes}{\pDi{q+1}{s}{}{}}
\newcommand{\Dqt}{\pDi{q}{t}{}{}}
\newcommand{\Dqpet}{\pDi{q+1}{t}{}{}}
\newcommand{\Dqn}[1]{\pDi{q}{#1}{0}{}}

\newcommand{\Dqtn}{\pDi{q}{t}{0}{}}

\newcommand{\Dqom}[1]{\pDq{#1}(\Omega)}
\newcommand{\Dqpeom}[1]{\Dqpe{#1}(\Omega)}
\newcommand{\Dqsom}{\Dqs(\Omega)}
\newcommand{\Dqpesom}{\Dqpes(\Omega)}
\newcommand{\Dqtom}{\Dqt(\Omega)}
\newcommand{\Dqpetom}{\Dqpet(\Omega)}

\newcommand{\Dqtnom}{\Dqtn(\Omega)}

\newcommand{\Dqpekmeh}{\Dqpe{<\meh}}
\newcommand{\Dqpekmehom}{\Dqpekmeh(\Omega)}

\newcommand{\dH}[3]{{}_{#3}\calH^{#1}_{#2}}
\newcommand{\dhqom}{\dH{q}{}{}(\Omega)}

\newcommand{\dhqepsom}[1]{\dH{q}{#1}{\eps}(\Omega)}

\newcommand{\skp}[2]{\langle#1,#2\rangle}

\newcommand{\skpb}[2]{\big\langle#1,#2\big\rangle}

\newcommand{\norm}[1]{|\normabst|#1|\normabst|}
\newcommand{\normb}[1]{\big|\normabst\big|#1\big|\normabst\big|}

\renewcommand{\FdGr}{(F_d,G_r)}

\newcommand{\dhqeps}{{\dH{q}{}{\eps}}}

\newcommand{\bD}[3]{{}_{#3}\mathbb{D}^{#1}_{#2}}
\newcommand{\bR}[3]{{}_{#3}\overset{\circ}{\mathbb{R}}{}^{#1}_{#2}}

\newcommand{\bDom}[3]{{}_{#3}\mathbb{D}^{#1}_{#2}(\Omega)}
\newcommand{\bRom}[3]{{}_{#3}\overset{\circ}{\mathbb{R}}{}^{#1}_{#2}(\Omega)}
\newcommand{\bDqsom}{\bD{q}{s}{0}(\Omega)}
\newcommand{\bRqsom}{\bR{q}{s}{0}(\Omega)}

\newcommand{\bRqpesom}{\bR{q+1}{s}{0}(\Omega)}

\newcommand{\bWom}[2]{\mathbb{W}^{#1}_{#2}(\Omega)}
\newcommand{\bWqsom}{\bWom{q}{s}}

\newcommand{\bWqom}{\bWom{q}{}}
\newcommand{\bWqpeom}{\bWom{q+1}{}}

\usepackage{algorithm,algorithmic}

\newcommand{\gt}{\gamma_\tau}

\newcommand{\chgt}{\check{\gamma}_\tau}

\newcommand{\xgn}{\gamma_n}
\newcommand{\xxrq}{\cR^q}
\newcommand{\xxrqpe}{\cR^{q+1}}
\newcommand{\El}{E_\lambda}
\newcommand{\Loesom}{\cS_\omega}
\newcommand{\pEms}{\tilde{E}_m}
\newcommand{\pHms}{\tilde{H}_m}
\newcommand{\EHms}{(\pEms,\pHms)}
\newcommand{\Fms}{\tilde{F}_m}
\newcommand{\Gms}{\tilde{G}_m}
\newcommand{\FGms}{(\Fms,\Gms)}

\newcommand{\paulydissregsatzausseneins}{\cite[Satz 3.6]{paulydiss}}
\newcommand{\paulydissregkorausseneins}{\cite[Korollar 3.8 (i)]{paulydiss}}
\newcommand{\paulydisskapvier}{\cite[Kapitel 4]{paulydiss}}
\newcommand{\paulydisskorsiebenfuenf}{\cite[Korollar 7.5]{paulydiss}}

\newcommand{\paulystaticsecgenstatic}{\cite[section 4]{paulystatic}}
\newcommand{\paulystaticgenstatictheo}{\cite[Theorem 4.6]{paulystatic}}
\newcommand{\paulystaticintdirichlet}{\cite[Lemma 3.8]{paulystatic}}
\newcommand{\paulystaticstatloesinhom}{\cite[Theorem 6.1, Remark 6.2]{paulystatic}}

\newcommand{\paulydecotrivdeco}{\cite[Lemma 5.1]{paulydeco}}
\newcommand{\paulydecodecoknh}{\cite[Theorem 3.2 (iv)]{paulydeco}}

\begin{document}

\date{2006}
\maketitle{}

\begin{abstract}
We discuss the radiation problem of total reflection for a time-harmonic generalized Maxwell system
in a nonsmooth exterior domain $\Omega\subset\mathbb{R}^N$\,, $N\geq3$\,, with nonsmooth inhomogeneous, anisotropic coefficients
converging near infinity with a rate $r^{-\tau}$\,, $\tau>1$\,, towards the identity.
By means of the limiting absorption principle a Fredholm alternative holds true and the
eigensolutions decay polynomially resp. exponentially at infinity.
We prove that the corresponding eigenvalues do not accumulate even at zero.
Then we show the convergence of the time-harmonic solutions to a solution of an
electro-magneto static Maxwell system as the frequency tends to zero.
Finally we are able to generalize these results easily to the corresponding Maxwell system with inhomogeneous boundary data.
This paper is thought of as the first and introductory one in a series of three papers,
which will completely discover the low frequency behavior of the solutions
of the time-harmonic Maxwell equations.\\
\keywords{Maxwell's equations, exterior boundary value problems, radiating solutions,
polynomial and exponential decay of eigensolutions, variable coefficients, electro-magneto static,
electro-magnetic theory, low frequency asymptotics, inhomogeneous boundary data}\\
\amsclass{35Q60, 78A25, 78A30}
\end{abstract}

\tableofcontents

\section{Introduction}

If we choose a time-harmonic ansatz (resp. Fourier transform with respect to time) for the classical time dependent Maxwell system in $\rd$
\begin{align*}
-\curl\mathbf{H}+\p_t\mathbf{D}&=\mathbf{I}&&,&
\curl\mathbf{E}+\p_t\mathbf{B}&=\mathbf{0}&&,\\
\pdiv\mathbf{D}&=\mbox{\boldmath$\rho$}&&,&
\pdiv\mathbf{B}&=\mathbf{0}&&,
\end{align*}
we are led to consider the time-harmonic Maxwell system with non zero complex frequency $\omega$
and complex valued data $\eps$\,, $\mu$\,, $I$ and $\rho$
\begin{align}
-\curl H+\ie\omega\eps E&=I&&,&
\curl E+\ie\omega\mu H&=0&&,\mylabel{Maxglcurl}\\
\pdiv\eps E&=\rho&&,&
\pdiv\mu H&=0&&.\mylabel{Maxgldiv}
\end{align}
This ansatz may be justified by the principle of limiting amplitude introduced by Eidus in \cite{eiduslampl}.
Here we denote the electric resp. magnetic field by $E$ resp. $H$\,, the displacement current resp. magnetic induction by
$D=\eps E$ resp. $B=\mu H$ and the current resp. charge density by $I$ resp. $\rho$\,.
The matrix valued functions $\eps$ and $\mu$ are assumed to be time independent and describe material properties,
i.e. the dielectricity and permeability of the medium. $\curl=\nabla\times\,$ (rotation) and $\pdiv=\nabla\,\cdot\,$ (divergence)
mark the usual differential operators from classical vector analysis.
By differentiation we get
$$\pdiv\eps E=-\frac{\ie}{\omega}\pdiv I\qqtext{,}\pdiv\mu H=0$$
from \eqref{Maxglcurl}, such that we can neglect (for $\omega\neq0$) the equations \eqref{Maxgldiv}.
To formulate these equations as a boundary value problem in
a domain $\Omega\subset\rd$ we need a boundary condition at $\p\Omega$\,.
Modeling total reflection of the electric field at the boundary, i.e. $\rN\ohne\Omega$ is a perfect conductor,
we impose the homogeneous boundary condition (assuming sufficient smoothness of the
boundary for the purpose of these introductory remarks)
\beq\nu\times E=0\qqtext{on}\p\Omega\qquad,\mylabel{MaxRandBed}\eeq
which means that $E$ possesses vanishing tangential components at $\p\Omega$\,.
Here $\nu$ denotes the outward unit normal on $\p\Omega$ and $\times$ the vector product in $\rd$\,. We are interested in the case of
an exterior domain $\Omega$\,, i.e. a connected open set with compact complement.
Therefore we have to impose an additional condition like
\beq\xi\times H+E\,,\,\xi\times E-H=o(r^\me)\mylabel{MaxStrahlBed}\eeq
\big($\xi(x):=x/|x|$\,, $r(x):=|x|$\big)
the classical so called outgoing Silver-M\"uller radiation condition, which allows to separate outgoing from incoming waves.
Interchanging $+$ and $-$ in \eqref{MaxStrahlBed} would yield incoming waves. We call the problem of finding $E$ and $H$
with \eqref{Maxglcurl}, \eqref{MaxRandBed} and \eqref{MaxStrahlBed} the radiation problem of total reflection for the time-harmonic
Maxwell system.

In 1952 Hermann Weyl \cite{weyl} suggests a generalization of the system \eqref{Maxglcurl} and \eqref{MaxRandBed} on Riemannian manifolds $\Omega$
of arbitrary dimension $N$ with the aid of alternating differential forms. If $E$ is a form of rank $q$ ($q$-form)
and $H$ a $(q+1)$-form and if we denote the exterior differential $\pd$
resp. the codifferential $\delta$ \big(acting on $q$- resp. $(q+1)$-forms\big) by
$$\rot:=\pd\qqtext{resp.}\pdiv:=\delta=(-1)^{qN}*\pd*$$
to remind of the electro-magnetic background
($*$: Hodge star-operator), the generalization of our system \eqref{Maxglcurl} and \eqref{MaxRandBed} reads
\begin{align}
\pdiv H+\ie\omega\eps E&=F&&,&\rot E+\ie\omega\mu H&=G&&,\mylabel{Maxglrotdiv}\\
\iota^* E&=0\mylabel{MaxRandBedz}
\end{align}
and we call it the generalized time-harmonic Maxwell system of total reflection.
Now $F$ (former $I$) is a $q$-form, $G$ (former $0$) a $(q+1)$-form,
$\eps$ resp. $\mu$ a linear transformation on $q$- resp. $(q+1)$-forms, $\iota:\p\Omega\hookrightarrow\ol{\Omega}$ the natural
embedding and $\iota^*$ the pull-back of $\iota$\,.
In the case $N=3$ and $q=1$\,, i.e. $E$ is a $1$-form and $H$ a $2$-form, the generalized Maxwell system is
equivalent to the classical Maxwell system of a perfect conductor, since the operators $\rot=\pd$ and $\pdiv=\delta$ acting on $q$-forms
are nothing else than the classical differential operators $\curl$ and $\pdiv$ if $q=1$ resp. $\pdiv$ and $-\curl$ if $q=2$\,.
Moreover, for $N=3$ and $1$- resp. $2$-forms $E$ we observe that the boundary condition \eqref{MaxRandBedz} means in the classical language
$\nu\times E=0$ resp. $\nu\cdot E=0$ on the boundary, i.e. vanishing tangential resp. normal components of the considered fields.
We remark that another classical case is discussed by this generalization. If $N=3$ and $q=0$ resp. $q=2$\,, i.e. $E$ resp. $H$ are
scalar valued, we get the equations of linear acoustics with homogeneous Dirichlet- resp. Neumann boundary condition, because
$\rot=\pd$ resp. $\pdiv=\delta$ turns out to be the classical gradient $\nabla$ on $0$- resp. $3$-forms.
Moreover, $\rot=\pd$ resp. $\pdiv=\delta$ is the zero-mapping on $3$- resp. $0$-forms.
In the case of an exterior domain $\Omega\subset\rN$\,, which we want to treat in this paper,
we give a generalization of the radiation condition \eqref{MaxStrahlBed} later. For a short notation we introduce
the formal matrix operators
\beq M:=\zmat{0}{\pdiv}{\rot}{0}\qqtext{,}\Lambda:=\zmat{\eps}{0}{0}{\mu}\mylabel{MLambdaDEF}\eeq
acting on pairs of $q$-$(q+1)$-forms and write our problem \eqref{Maxglrotdiv}, \eqref{MaxRandBedz} easily as
\beq(M+\ie\omega\Lambda)\EH=(F,G)\qqtext{,}\iota^*E=0\qquad.\mylabel{shortproblem}\eeq
\big(For typographical reasons we write form-pairs as $\EH$\,, although the matrix calculus would expect the notation
$\zvec{E}{H}$\,.\big)

Time-harmonic exterior boundary value problems concerning the classical Maxwell equations, i.e. $N=3$ and $q=1$\,,
have been studied by M\"uller \cite{mueller} in domains with smooth boundaries and homogeneous, isotropic media, i.e. $\eps=\mu=\id$\,,
with integral equation methods and by Leis \cite{leistheo} \big(see also \cite{leisbuch}\big)
with the aid of the limiting absorption principle for media, which are inhomogeneous and anisotropic within a bounded subset of $\Omega$\,.
The generalized time-harmonic Maxwell system has been treated by Weck \cite{weckhabil} and Picard \cite{picardhabil}.

In this paper we want to discuss the time-harmonic radiation boundary value problem of total reflection for the generalized
Maxwell equations \eqref{shortproblem} in an exterior domain $\Omega$ of $\rN$ for arbitrary dimensions $N$ and ranks $q$\,.
A main goal of our investigations is to treat data $\FG$ in weighted $\lzom$-spaces and inhomogeneous, anisotropic and irregular
\big($\text{L}^\infty(\Omega)$-\big) coefficients $\eps$\,, $\mu$ converging near infinity with a rate $r^{-\tau}$\,, $\tau>0$\,, towards the identity.
\big($r(x):=|x|$ denotes the Euclidean norm in $\rN$\,.\big)
We follow in close lines the papers of Weck and Witsch \cite{linelae} and
Picard,  Weck and Witsch \cite[part 1]{xmas}, which deal with the system of generalized linear elasticity and the classical
Maxwell equations. In particular we generalize the results obtained in the second paper to arbitrary dimensions $N$ and ranks of
forms $q$\,. To present a time-harmonic solution theory we prove that for nonzero frequencies $\omega$ and data
$\FG\in\Lzqgehom\times\Lzqpegehom$\footnote{The Definitions will be supplied in section \ref{defsection}.}
and $\text{L}^\infty$-coefficients $\eps$\,, $\mu$ a Fredholm alternative holds true.
The main tool to handle irregular coefficients is a decomposition lemma, which allows
us to prove the polynomial decay of eigensolutions as well as an a-priori estimate needed to establish the validity
of the limiting absorption principle by reduction to the similar results known for the scalar Helmholtz equation.
The key to this decomposition lemma are weighted Hodge-Helmholtz decompositions, i.e. decompositions in
irrotational and solenoidal fields, in the whole space case, which have been proved in \cite{sphharm}.

The idea of the decomposition lemma is to use a well known procedure to decouple the electric and magnetic field by discussing a second
order elliptic system.
To illustrate this calculation let us look at \eqref{shortproblem} in the homogeneous case $\Lambda=\id$\,.
Applying $M-\ie\omega$ yields
\beq(M^2+\omega^2)\EH=(M-\ie\omega)\FG\qquad.\mylabel{calcsecorderz}\eeq
If we choose $F$ solenoidal, i.e. $\pdiv F=0$\,, and $G$ irrotational, i.e. $\rot G=0$\,, these properties will be
transfered to $E$\,, i.e. $\pdiv E=0$\,, and $H$\,, i.e. $\rot H=0$\,, by \eqref{shortproblem} because of
$$\pdiv\pdiv=0\qqtext{and}\rot\rot=0\qquad.$$
From $\Delta=\rot\pdiv+\pdiv\rot$\,, where the Laplacian acts on each Euclidean component,
we get $M^2\EH=(\pdiv\rot E,\rot\pdiv H)=\Delta\EH$ and
finally \eqref{calcsecorderz} turns to the (componentwise) Helmholtz equation
\beq(\Delta+\omega^2)\EH=(M-\ie\omega)\FG\qquad.\mylabel{calcsecorderd}\eeq
Armed with the polynomial decay of eigensolutions and an a-priori estimate for the solutions corresponding to non-real frequencies
(We get these solutions from the existence of a selfadjoint realization of $M$\,.) we obtain our radiating
solutions for frequencies $\omega\in\rzon$ with the method of limiting absorption invented by Eidus \cite{eidusla}
as limits of solutions for frequencies $\omega\in\czp\ohne\rz$\,.
We have to admit finite dimensional eigenspaces for certain eigenvalues but show that these possibly existing
eigenvalues do not accumulate in $\rzon$\,.
All these results can be proved by the techniques used in \cite{xmas} and for orders of decay $\tau>1$\,.
Thus we do not want to repeat them in this paper.
However, we refer the interested reader to \paulydisskapvier\, for the detailed proofs.

Proving an estimate for the solutions of the homogeneous, isotropic whole space problem
with the aid of a representation formula and studying some special convolution kernels (Hankel functions) we even can exclude 0 as an
accumulation point of eigenvalues. Thus the time-harmonic solution operator $\loesom$ is well defined on $\Lzqgehom\times\Lzqpegehom$
for small frequencies $\omega\neq0$\,.
To reach this aim we have to increase the order of decay of the coefficients $\eps-\id$\,, $\mu-\id$
to $\tau>(N+1)/2$ and assume that they are $\pc{1}$ in the outside of an arbitrarily large ball.
Assuming stronger differentiability assumptions on $\eps-\id$ and $\mu-\id$\,,
i.e. $\pc{2}$ in the outside of a ball,
we are able to show the exponential decay of eigensolutions as well. To the best of our knowledge it is an open question whether there
exist such eigenvalues in this general case. Recently under comparable stronger assumptions on the coefficients
Bauer \cite{bauerabsence} was able to prove that no eigenvalues occur in the classical case of Maxwell equations ($N=3$\,, $q=1$).
Unfortunately his methods are not applicable in our general case. It seems to be the same problem that arises trying to
prove the principle of unique continuation for the generalized Maxwell equation. In the classical case
the principle of unique continuation was shown by Leis \cite{leispef} or \cite[p. 168, Theorem 8.17]{leisbuch}.
However, in the case of homogeneous, isotropic coefficients, i.e. $\eps=\id$\,, $\mu=\id$\,,
in the outside of a ball
all components of a possible eigensolution solve the homogeneous Helmholtz equation \big(compare \eqref{calcsecorderd}\big)
near infinity and therefore by Rellich's estimate \cite{rellich} must have compact support. With the validity of the
principle of unique continuation for our Maxwell system this eigensolution must vanish.
In the general case the principle of unique continuation is valid for scalar valued $\pc{2}$-functions $\eps$\,, $\mu$
and in the classical case for matrices $\eps$\,, $\mu$ with entries in $\pc{2}$\,. (See the citation above from Leis.)

Having established the time-harmonic solution theory in section \ref{time-harmonicsection}
we approach the low frequency asymptotics of our time-harmonic solution operator. To this end first we
have to provide a static solution theory. This one is more complicated than for example the static solution theory for
Helmholtz' equation. The first reason is that for $\omega=0$ the system \eqref{Maxglrotdiv} resp. \eqref{shortproblem}, i.e.
$$\rot E=G\qqtext{,}\pdiv H=F\qquad,$$
is no longer coupled and that we have to add two more equations to determine $E$ and $H$\,, i.e.
\beq\pdiv\eps E=f\qqtext{,}\rot\mu H=g\qquad,\mylabel{Maxgladd}\eeq
which in the case $\omega\neq0$ automatically follow by differentiation from \eqref{Maxglrotdiv} as mentioned above.
\big($f=-\frac{\ie}{\omega}\pdiv F$ and $g=-\frac{\ie}{\omega}\rot G$\,, if $\pdiv F$ and $\rot G$ exist.\big)
Furthermore, we need a boundary condition for the magnetic field (form). Because $\rot=\pd$ and $\iota^*$ commute
we derive $\ie\omega\iota^*\mu H=\iota^*G$ for $\omega\neq0$ from \eqref{Maxglrotdiv}.
This suggests to impose a condition
on the term $\iota^*\mu H$ and, for example, we can choose the homogeneous boundary condition
$$\iota^*\mu H=0$$
for our magnetic field. The second reason is that this static Maxwell boundary value problem
\begin{align}
\rot E&=G&&,&\pdiv H&=F&&,\non\\
\pdiv\eps E&=f&&,&\rot\mu H&=g&&,\mylabel{staticMaxsystem}\\
\iota^* E&=0&&,&\iota^*\mu H&=0\non
\end{align}
has a nontrivial kernel $\dhqepsom{}\times\mu^\me\dH{q+1}{}{\mu^\me}(\Omega)$ consisting of harmonic Dirichlet forms.
Thus we are forced to work with orthogonality constraints on the static solutions to achieve uniqueness.
For the static system \eqref{staticMaxsystem} a solution theory was given by Kress \cite{kress} and
Picard \cite{potential} for the homogeneous, isotropic case, i.e. $\eps=\id$\,, $\mu=\id$\,,
by Picard \cite{decomposition} for the inhomogeneous, anisotropic case (Here $\eps$ and $\mu$ even are
allowed to be nonlinear transformations.) as well as by Picard \cite{boundaryelectro} for the inhomogeneous, anisotropic classical case.
For our purpose we need a result like that given by Picard in \cite{potential}.
In \cite{paulystatic} we will discuss the electro-magneto static problem with inhomogeneous,
anisotropic coefficients $\eps$\,, $\mu$ in detail.
We shortly present some of these results and introduce our static solution concept in section \ref{staticsection}.

Then in section \ref{asymsection}, the main section of this paper,
we prove the convergence of the time-harmonic solutions to a special static solution of \eqref{staticMaxsystem}.
This result generalizes the paper of Picard \cite{asymmax}, which considers the classical Maxwell equations,
to arbitrary odd dimensions $N$ and ranks $q\neq0$ as well as to coefficients and right hand side data,
which necessarily do not have to be compactly supported.
We note that similar results hold true for even dimensions. Since the complexity of the calculations increases
considerably due to the appearance of logarithmic terms in the fundamental solution (Hankel's function),
we restrict our considerations to odd dimensions.

The last section \ref{inhomo} deals with inhomogeneous boundary conditions. Using a new result from Weck \cite{wecklip},
which allows to define traces of $q$-forms on domains with Lipschitz-boundaries, we discuss the time-harmonic problem
\begin{align*}
(M+\ie\omega\Lambda)\EH&=\FG&&,&\iota^*E&=\lambda
\intertext{and the static problem}
M\EH&=\FG&&,&(\iota^*E,\iota^*\mu H)&=(\lambda,\varkappa)&&,\\
(\pdiv\eps E,\rot\mu H)&=\fg&&.
\end{align*}
It turns out that the solution theories as well as the low frequency asymptotics for these problems are easy consequences
of the results for homogeneous boundary conditions and the existence of an adequate extension operator for our traces.

Easily by the Hodge star-operator we always get the corresponding dual results, but we renounce them to shorten this paper.

Essentially this is the first part of the authors ph.~d.~thesis. Thus sometimes we only sketch or neglect some proofs
and do not mention all results obtained in \cite{paulydiss}.
To get more details on the proofs or some additional results we refer the interested reader to \cite{paulydiss}.

This paper is the first one in a series of three papers having the aim to determine the low
frequency asymptotics of the solutions of the time-harmonic Maxwell equations completely.
In the second paper \cite{paulystatic} we will discuss the corresponding electro-magneto static equations in detail and
show, how one may define powers of a static solution operator in weighted Sobolev spaces.
This allows us to write down a generalized Neumann sum, which is a good candidate for the
asymptotic series approaching the time-harmonic solutions for small frequencies.
In the third paper we finally present the complete low frequency asymptotics in the operator norm of
weighted Sobolev spaces up to arbitrary orders in powers of the frequency.

\section{Definitions and preliminaries}

We will consider an exterior domain $\Omega\subset\rN$\,, i.e. $\rN\ohne\Omega$ is compact,
as a special Riemannian manifold of dimension $3\leq N\in\nz$\,. \mylabel{defsection}

We fix a radius $r_0$ and some radii $r_n:=2^nr_0$\,, $n\in\nz$\,, such that $\rN\ohne\Omega$ is a compact subset of $U_{r_0}$\,,
the open ball with radius $r_0$ centered at the origin.
For later purpose we choose a cut-off function $\mbox{\boldmath$\eta$}$\,, such that \mylabel{paulysecdef}
\beq\mbox{\boldmath$\eta$}\in\cu(\rz,\rz)\qqtext{,}\supp\mbox{\boldmath$\eta$}\subset[1,\infty)\qqtext{,}\restr{\mbox{\boldmath$\eta$}}{[2,\infty)}=1\qquad,\mylabel{etafettdef}\eeq
and define two other cut-off functions by
\beq\hat{\eta}(t):=\mbox{\boldmath$\eta$}\big(1+\frac{t-r_1}{r_2-r_1}\big)\mylabel{etadachdef}\eeq
and
\beq\eta:=\hat{\eta}\circ r\qquad.\mylabel{etadef}\eeq
Setting $A_r:=\rN\ohne\ol{U_r}$ and $Z_{r,\tilde{r}}:=A_{r}\cap U_{\tilde{r}}$
we note $\supp\nabla\eta\subset\ol{Z_{r_1,r_2}}$\,.

Using the weight function
$$\rho:=(1+r^2)^{1/2}$$
we introduce for $m\in\nzn$ and $s\in\rz$ the weighted Sobolev spaces
\begin{align*}
\hmsom&:=\setb{u\in\Lzlocom}{\rho^{s+|\alpha|}\p^\alpha u\in\lzom\text{ for all }|\alpha|\leq m}\qquad,\\
\subset\Hmsom&:=\setb{u\in\Lzlocom}{\rho^s\p^\alpha u\in\lzom\text{ for all }|\alpha|\leq m}\qquad.
\end{align*}
Equipped with their natural norms these are clearly Hilbert spaces.
In the special cases $m=0$ or $s=0$ we also write
\begin{align*}
\hmom&:=\hom{m}{0}&&,&\Hmom&=\Hom{m}{0}&&,\\
\Lzsom&:=\hom{0}{s}=\Hom{0}{s}&&,&\lzom&=\hom{0}{0}=\Hom{0}{0}&&.
\end{align*}

In $\om$ we have a global chart, the identity, and thus naturally $\Omega$ becomes a $N$-dimensional smooth Riemannian manifold with Cartesian
coordinates $\{x_1,\dots,x_N\}$\,. For alternating differential forms of rank $q\in\zz$ ($q$-forms)
we define componentwise partial derivatives $\p^\alpha u=(\p^\alpha u_I)\pd x^I$\,, if $u=u_I\pd x^I$ (sum convention!),
where $I$ are ordered multi-indices of length $q$\,, and introduce for $m\in\nzn$ and $s\in\rz$
the weighted Sobolev spaces $\hqmsom$ resp. $\Hqmsom$ of $q$-forms.
(Clearly we use the natural componentwise norms in these Hilbert spaces.)
Again in the cases $m=0$ or $s=0$ we use the same abbreviations as in the scalar case.
Especially for $m=s=0$ and $f=f_I\pd x^I,g=g_I\pd x^I\in\Lzqom{}$ we have the scalar product
$$\skp{f}{g}_{\Lzqom{}}=\intom f\wedge*\ol{g}=\intom*\skp{f}{g}_q=\intom\skp{f}{g}_q\,d\lambda=\intom f_I\ol{g}_I\,d\lambda\qquad.$$
($\lambda$\,: Lebesgue-measure, $\skp{\,\cdot\,}{\,\cdot\,}_q$\,: pointwise scalar product, $*$\,: Hodge star-operator)

Throughout this paper we denote the exterior derivative
$\pd$ by $\rot$ and the co-derivative $\delta=\pm*\pd*$ by $\pdiv$
to remind of the electro-magnetic background. Because of Stokes' theorem and the product rule
on $\cqunom$ \big(the vector space of all smooth $q$-forms with compact support in $\om$\big)
these linear operators are formally skew adjoint to each other, i.e.
\beq\skp{\rot\Phi}{\Psi}_{\lzqpeom}=-\skp{\Phi}{\pdiv\Psi}_{\lzqom}\qquad\qquad\forall\quad\Phi,\Psi\in\cqunom\qquad,\mylabel{skewadj}\eeq
which gives rise to weak definitions of $\rot$ and $\pdiv$\,. We note that still $\rot\rot=0$\,,
$\pdiv\pdiv=0$ and $\rot\pdiv+\pdiv\rot=\Delta$ hold true in the weak sense.
Furthermore, for $s\in\rz$ we need some special weighted spaces suited for Maxwell's equations:
\begin{align*}
\rqsom&:=\setb{E\in\Lzqsom}{\rot E\in\Lzqpeom{s+1}}\\
\subset\Rqsom&:=\setb{E\in\Lzqsom}{\rot E\in\Lzqpesom}\qquad,\\
\dqsom&:=\setb{H\in\Lzqsom}{\pdiv H\in\qLzom{q-1}{s+1}}\\
\subset\Dqsom&:=\setb{H\in\Lzqsom}{\pdiv H\in\qLzom{q-1}{s}}
\end{align*}
Equipped with their natural graph norms these are all Hilbert spaces. To generalize the homogeneous
boundary condition we introduce $\ronqsom$ resp. $\Ronqsom$ as the closure of $\cqun(\Omega)$ in
the corresponding graph norm $\norm{\,\cdot\,}_{\rqsom}$ resp. $\norm{\,\cdot\,}_{\Rqsom}$\,.
Using Stokes' theorem we see that in fact the homogeneous boundary condition $\iota^*E=0$ is generalized
in these spaces. The spaces $\Rqsom$\,, $\Dqsom$ and even $\Ronqsom$ are invariant under multiplication
with bounded smooth functions $\varphi$\,, i.e. for $E\in\Rqsom$ we compute
$$\rot(\varphi E)=(\rot\varphi)\wedge E+\varphi\,\rot E\qquad.$$
A subscript $0$ at the lower left corner
indicates vanishing rotation resp. divergence, e.g. $\ronqsnom=\setb{E\in\ronqsom}{\rot E=0}$\,,
and in the special case $s=0$ we neglect the weight index, e.g. $\dqnom{}:=\dqnom{0}$\,.
If we consider the whole space, i.e. $\Omega=\rN$\,, we omit the dependence on the domain and write for example
$\rqsn:=\rqsn(\rN)$\,. For every weighted Sobolev spaces $V_t$\,, $t\in\rz$\,, we define
$$V_{<s}:=\bigcap_{t<s}V_t\qqtext{,}V_{>s}:=\bigcup_{t>s}V_t\qquad,$$
e.g. $\Dqom{<\meh}$\,. Moreover, replacing a weight index $s$ by the symbol $s=\loc$ resp. $s=\vox$ means that the forms
are locally square-integrable at infinity but square-integrable up to the boundary resp. that the forms have bounded supports.

Furthermore, the rule of partial integration \eqref{skewadj} may be generalized as follows: Using a usual cutting technique
we get for
$$\EH\in\Ronqtom\times\Dqpesom\qqtext{resp.}\EH\in\ronqtom\times\dqpesom$$
with $t,s\in\rz$ and $t+s\geq0$ resp. $t+s\geq-1$
\beq\skp{\rot E}{H}_{\lzqpeom}+\skp{E}{\pdiv H}_{\lzqom}=0\qquad.\mylabel{skewadjtwo}\eeq
Now let us introduce our transformations:

\begin{defini}\mylabel{transdefione}
Let $\tau\geq0$\,. We call a transformation $\eps$ $\tau$-{\sf admissible}, if
\begin{itemize}
\item $\eps(x)$ is a linear transformation on $q$-forms for all $x\in\om$\,,
\item $\eps$ possesses $\text{\rm L}^\infty(\Omega)$-coefficients, i.e. the matrix representation of
$\eps$ corresponding to the canonical basis \big(and then for every chart basis $\{\pd h^I\}$\big)
has $\text{\rm L}^\infty(\Omega)$-entries,
\item $\eps$ is symmetric, i.e. for all $E,H\in\lzqom$
$$\skp{\eps E}{H}_{\lzqom}=\skp{E}{\eps H}_{\lzqom}$$
holds, and uniformly positive definite, i.e.
$$\bigvee_{c>0}\bigwedge_{E\in\lzqom}\skp{\eps E}{E}_{\lzqom}\geq c\cdot\norm{E}_{\lzqom}^2\qquad,$$
\item $\eps$ is asymptotically the identity, i.e.
$\eps=\epsn\id+\hat{\eps}$ with $\epsn\in\rzp$ and $\hat{\eps}=\calO(r^{-\tau})$ as $r\to\infty$\,.
We call $\tau$ the `{\it order of decay}' of the perturbation $\hat{\eps}$\,.
\end{itemize}
\end{defini}

For some results obtained in this paper we need one more additional assumption on the perturbations
$\epsd$ of our transformations. That is $\epsd$ has to be differentiable in the outside of
an arbitrarily large ball. More precisely:

\begin{defini}\mylabel{transdefitwo}
Let $\tau\geq0$\,. We call a transformation $\eps$ $\tau$-$\pc{1}$-{\sf admissible}, if
\begin{itemize}
\item $\eps$ is $\tau$-admissible
\item and $\epsd\in\pc{1}(A_{r_0})$\,, which means that the matrix representation of $\epsd$ corresponding to the canonical basis
\big(and then for every chart basis $\{\pd h^I\}$\big) has $\pc{1}(A_{r_0})$-entries, with the additional asymptotic
$$\p_n\epsd=\calO(r^{-1-\tau})\qqtext{as}r\to\infty\qqtext{,}n=1,\dots,N\qquad.$$
\end{itemize}
\end{defini}

Moreover, we need a special property of our boundary $\p\om$\,:

\begin{defini}\mylabel{maxkompakt}
A bounded domain $\Xi$ possesses the `Maxwell compactness property', shortly {\sf MCP},
if and only if the embeddings
$$\Ronq{}(\Xi)\cap\pDq{}(\Xi)\hookrightarrow\Lzq{}(\Xi)$$
are compact for all $q$\,.
\end{defini}

The {\sf MCP} is a property of the boundary and there is a great amount of literature about the {\sf MCP}.
The first idea was to estimate the $\qH{1}{q}{}{}(\Xi)$-norm
by the $\big(\pRq{}(\Xi)\cap\pDq{}(\Xi)\big)$-norm (Gaffney's inequality) and then to use Rellich's selection theorem.
To do this one needs smooth boundaries, which, for instance, may be seen in \cite[p. 157, Theorem 8.6]{leisbuch}.
If $q=0$ we even have
$$\pR{0}{}{}{\circ}(\Xi)\cap\pDi{0}{}{}{}(\Xi)=\pR{0}{}{}{\circ}(\Xi)=\qH{1}{0}{}{\circ}(\Xi)\qquad.$$
In 1972 \cite{weckhabil,weckmax} Weck presented for the first time a proof of the {\sf MCP} for bounded manifolds with nonsmooth boundaries
(`cone-property'). More proofs of the {\sf MCP} were given by Picard \cite{comimb} (`Lipschitz-domains') and in the classical
case by Weber \cite{weber} (another `cone-property') and Witsch \cite{witsch} (`$p$-cusp-property).
A proof of the {\sf MCP} in the classical case for bounded domains handling
the largest known class of boundaries has been given by
Picard, Weck and Witsch in \cite{xmas}. They combined the techniques from \cite{weckmax,comimb,witsch}.

\begin{defini}\mylabel{maxkompaktlokal}
$\Omega$ possesses the `Maxwell local compactness property', shortly {\sf MLCP},
if and only if the embeddings
$$\Ronqom{}\cap\Dqom{}\hookrightarrow\Lzqloc(\omq)$$
are compact for all $q$\,.
\end{defini}

\begin{rem}\mylabel{maxkompaktlokalrem}
The following assertions are equivalent:
\begin{itemize}
\item[\rm\bf (i)] $\om$ possesses the {\sf MLCP}.
\item[\rm\bf (ii)] $\om\cap U_\varrho$ possesses the {\sf MCP} for all $\varrho\geq r_0$\,.
\item[\rm\bf (iii)] The embeddings
$$\Ronqsom\cap\Dqsom\hookrightarrow\Lzqtom$$
are compact for all $t,s\in\rz$ with $t<s$ and all $q$\,.
\item[\rm\bf (iv)] For all $t,s\in\rz$ with $t<s$\,, all $q$ and all $0$-admissible $\eps_q$ the embeddings
$$\Ronqsom\cap\eps_q^\me\Dqsom\hookrightarrow\Lzqtom$$
are compact.
\end{itemize}
\end{rem}

Let $\eps$ be a $0$-admissible transformation and $t\in\rz$\,.
We introduce the `{\it (weighted harmonic) Dirichlet forms}'
\beq\dhqepsom{t}:=\ronqnom{t}\cap\eps^\me\dqnom{t}\mylabel{dirichletformen}\eeq
and in the special case $\eps=\id$ we denote them by $\dH{q}{t}{}(\Omega)$\,. If $t=0$\,,
we always write $\dhqepsom{}:=\dhqepsom{0}$\,.

By the projection theorem and the $\Lzqom{}$-orthogonality of
$\ol{\rot\pR{q-1}{}{}{\circ}(\Omega)}$ and $\dqnom{}$ resp. $\ol{\pdiv\Dqpeom{}}$ and $\ronqnom{}$
as well as the inclusions $\ol{\rot\pR{q-1}{}{}{\circ}(\Omega)}\subset\ronqnom{}$ and $\ol{\pdiv\Dqpeom{}}\subset\dqnom{}$
we get the following Helmholtz decompositions:
\begin{align}
\begin{split}
\lzqom&=\ol{\rot\pR{q-1}{}{}{\circ}(\Omega)}\oplus_\eps\eps^\me\dqnom{}=\ronqnom{}\oplus_\eps\eps^\me\ol{\pdiv\Dqpeom{}}\\
&=\eps^\me\ol{\rot\pR{q-1}{}{}{\circ}(\Omega)}\oplus_\eps\dqnom{}=\eps^\me\ronqnom{}\oplus_\eps\ol{\pdiv\Dqpeom{}}\\
&=\ol{\rot\pR{q-1}{}{}{\circ}(\Omega)}\oplus_\eps\dhqepsom{}\oplus_\eps\eps^\me\ol{\pdiv\Dqpeom{}}\\
&=\eps^\me\ol{\rot\pR{q-1}{}{}{\circ}(\Omega)}\oplus_\eps\eps^\me\dH{q}{}{\eps^\me}(\Omega)\oplus_\eps\ol{\pdiv\Dqpeom{}}
\end{split}\mylabel{lzzerlegungenaussen}
\end{align}
Here all closures are taken in $\lzqom$ and we denote the $\skp{\eps\,\cdot\,}{\,\cdot\,}_{\lzqom}$-orthogonality by
$\oplus_\eps$ and put $\oplus:=\oplus_{\id}$\,. These Helmholtz decompositions may be found in
\cite[Lemma 1]{potential}, \cite[Lemma 1]{decomposition} or in the
classical case in \cite[p. 168]{boundaryelectro}, \cite[Lemma 3.13]{xmas}.

If $\om$ possesses the {\sf MLCP} and $\eps$ is $\tau$-$\pc{1}$-admissible with $\tau>0$\,,
then \paulystaticintdirichlet\, shows
\beq\dhqepsom{-\Nh}=\dhqepsom{}=\dhqepsom{<\Nh-1}\mylabel{dirichletint}\eeq
and using the Helmholtz decompositions \eqref{lzzerlegungenaussen} we easily see
$$d^q:=\dim\dhqepsom{}=\dim\dhqom{}<\infty\qquad,$$
i.e. $d^q$ depends neither on weights $-N/2\leq t<N/2-1$ nor on the transformation $\eps$\,.

Finally we define three operators
\beq R:=r\pd r\wedge\,\cdot=x_n\pd x^n\wedge\,\cdot\qtext{,}T:=(-1)^{qN}*R*\qtext{,}S:=\zmat{0}{T}{R}{0}\mylabel{RTSdef}\eeq
acting pointwise on $q$- resp. $(q+1)$- resp. pairs of $q$- and $(q+1)$-forms, which will be useful to formulate
the radiation condition. These operators correspond to $\rot$\,, $\pdiv$ and $M$ in the following way:
If $\varphi$ is a smooth function, $E$ a $q$-form with weak rotation and $H$ a $(q+1)$-form with weak divergence, then
\begin{align}
\rot\big(\varphi(r)E\big)&=\varphi(r)\rot E+\varphi'(r)r^\me RE\qquad,\non\\
\pdiv\big(\varphi(r)H\big)&=\varphi(r)\pdiv H+\varphi'(r)r^\me TH\qquad,\mylabel{rotdivcom}\\
M\big(\varphi(r)\EH\big)&=\varphi(r)M\EH+\varphi'(r)r^\me S\EH\qquad.\non
\end{align}
There is another correspondence between these operators. If we define the Fourier transformation $\calF$ on $q$-forms in $\rN$
componentwise in Euclidean coordinates, then the mapping $\calF:\lzq\to\lzq$ is unitary and the well known formulas
$$\calF(\p^\alpha E)=\ie^{|\alpha|}\id^{\alpha}\calF(E)\qqtext{,}
\p^\alpha\calF(E)=(-\ie)^{|\alpha|}\calF(\id^{\alpha}E)$$
and clearly $\calF(\Delta E)=-r^2\calF(E)$\,, $\Delta\calF(E)=-\calF(r^2E)$ hold for $q$-forms $E$\,.
By elementary calculations we get
\begin{align}
\calF\rot&=\ie R\calF&&,&\calF\pdiv&=\ie T\calF&&,&\calF M&=\ie S\calF&&,\mylabel{FourierRTSone}\\
\rot\calF&=-\ie\calF R&&,&\pdiv\calF&=-\ie\calF T&&,&M\calF&=-\ie\calF S&&.\mylabel{FourierRTStwo}
\end{align}

\section{The time-harmonic problem}\label{time-harmonicsection}

Let $\tau\geq0$ and $\eps$\,, $\mu$ be two $\tau$-admissible transformations on $q$- resp. $(q+1)$-forms
as well as $M$\,, $\Lambda$ be as in \eqref{MLambdaDEF}.
As mentioned above we want to treat the time-harmonic, inhomogeneous, anisotropic (generalized)
Maxwell equation
$$(M+\ie\omega\Lambda)\EH=\FG\qquad,$$
with frequencies
$$\omega\in\czp:=\set{z\in\cz}{\imt z\geq0}\qquad.$$
A substitution like $\tilde{x}:=\alpha x$\,, $\tilde{H}:=\beta H$ allows us to suppose w. l. o. g.
\beq\epsn=\mun=1\qqtext{and thus}\Lambda=\id+\Lambdad\qquad.\mylabel{epsnmuneins}\eeq
To shorten and simplify the formulas we always want to assume \eqref{epsnmuneins} throughout this paper.

Now let us introduce our time-harmonic solution concept. From the skewadjointness of the two operators
\begin{align*}
\rot:\Ronqom{}\subset\lzqom&\To\lzqpeom\qquad,\\
\pdiv:\Dqpeom{}\subset\lzqpeom&\To\lzqom
\end{align*}
to each other we obtain the selfadjointness of
$$\calM:\Ronqom{}\times\Dqpeom{}\subset{}_\eps\lzqom\times{}_\mu\lzqpeom\To{}_\eps\lzqom\times{}_\mu\lzqpeom$$
with
$$\calM\EH:=\ie\Lambda^\me M\EH=\ie(\eps^\me\pdiv H,\mu^\me\rot E)\qquad.$$
Here ${}_\nu\lzqom:=\lzqom$ is equipped with the scalar product $\skp{\nu\,\cdot\,}{\,\cdot\,}_{\lzqom}$\,. This suggests

\begin{defini}\mylabel{defloesc}
Let $\omega\in\cz\ohne\rz$ and $\FG\in\lzqom\times\lzqpeom$\,. Then $\EH$ solves the problem $\Max(\Lambda,\omega,F,G)$\,, if and only if
\begin{itemize}
\item[\rm\bf (i)]\qquad $\EH\in\Ronqom{}\times\Dqpeom{}$\qquad,
\item[\rm\bf (ii)]\qquad $(M+\ie\omega\Lambda)\EH=\FG$\qquad.
\end{itemize}
\end{defini}

The selfadjointness of $\calM$ yields the unique solvability of $\Max(\Lambda,\omega,F,G)$
for each frequency $\omega\in\cz\ohne\rz$ and all $\FG\in\lzqom\times\lzqpeom$\,.
We denote the continuous solution operator by
$$\loesom:=\ie(\calM-\omega)^\me\Lambda^\me\qquad.$$
It can be seen easily that the spectrum of $\calM$ is the entire real axis.
Thus we expect, e.g. from Helmholtz' equation that we have to work in weighted $\lz$-spaces and with radiating solutions
to get a solution theory for real frequencies.

\begin{defini}\mylabel{defloesr}
Let $\omega\in\rz\ohne\{0\}$ and $\FG\in\Lzqlocom\times\Lzqpelocom$\,. Then $\EH$ solves the problem $\Max(\Lambda,\omega,F,G)$\,, if and only if
\begin{itemize}
\item[\rm\bf (i)]\qquad $\EH\in\Ronqkmehom\times\Dqpekmehom$\qquad,
\item[\rm\bf (ii)]\qquad $(M+\ie\omega\Lambda)\EH=\FG$\qquad,
\item[\rm\bf (iii)]\qquad $(r^\me S+\id)\EH\in\Lzqgmeh(\Omega)\times\Lzqpegmeh(\Omega)$\qquad.
\end{itemize}
\end{defini}

\begin{rem}\mylabel{defloesrbemeins}
We call condition (iii) the `Maxwell radiation condition' or `radiation condition'.
This condition generalizes the classical ($N=3$\,, $q=1$) Silver-M�ller incoming radiation condition
for Maxwell equations \big(see \eqref{MaxStrahlBed}\big)
$$\xi\times H-E\in\Lz{>\meh}(\Omega)\qqtext{,}\xi\times E+H\in\Lz{>\meh}(\Omega)\qquad.$$
We note that the radiation condition reads
$$(r^\me TH+E,r^\me RE+H)\in\Lzqgmeh(\Omega)\times\Lzqpegmeh(\Omega)\qquad.$$
\end{rem}

Furthermore, we need

\begin{defini}\mylabel{punktspektrum}
We define
\begin{align*}
\pP&:=\setb{\omega\in\cz\ohne\{0\}}{\Max(\Lambda,\omega,0,0)\;\;\text{\rm has a nontrivial solution.}}
\intertext{and for $\omega\in\cz\ohne\{0\}$}
\calN(\Max,\Lambda,\omega)&:=\setb{\EH}{\EH\;\;\text{\rm is a solution of }\Max(\Lambda,\omega,0,0)\,.}\quad.
\end{align*}
\end{defini}

Clearly we have $\pP\subset\rzon$ and
$\calN(\Max,\Lambda,\omega)=N(\calM-\omega)=\big\{(0,0)\big\}$
for $\omega\in\cz\ohne\rz$\,.

Similar arguments like those leading to the main result of the first part of \cite{xmas} prove the
following theorem. Therefore these do not have to be repeated here.
We note that essentially we need two a priori estimates. Then the time-harmonic solutions are
obtained by the limiting absorption principle.
For details we refer the interested reader to \paulydisskapvier.

\begin{theo}\mylabel{fredholm}
Let $\tau>1$ and $\omega\in\rzon$\,.
\begin{itemize}
\item[\rm\bf (i)] For all $t\in\rz$
\begin{align*}
\calN(\Max,\Lambda,\omega)&=N(\calM-\omega)\\
&\subset\big(\Ronqtom\cap\eps^\me\Dqtnom\big)\times\big(\Dqpetom\cap\mu^\me\Ronqpetnom\big)\quad,
\end{align*}
i.e. eigensolutions decay polynomially.
\end{itemize}
Additionally let $\Omega$ have the {\sf MLCP}. Then:
\begin{itemize}
\item[\rm\bf (ii)] $\calN(\Max,\Lambda,\omega)$ is finite dimensional.
\item[\rm\bf (iii)] $\pP$ has no accumulation point in $\rzon$\,.
\item[\rm\bf (iv)] For every $\FG\in\Lzqgehom\times\Lzqpegehom$ there exists a solution $\EH$
of the problem $\Max(\Lambda,\omega,F,G)$\,, if and only if
\beq\bigwedge_{\eh\in\calN(\Max,\Lambda,\omega)}\qquad\skpb{\FG}{\eh}_{\lzqom\times\lzqpeom}=0\qquad.\mylabel{fgsenkrecht}\eeq
The solution can be chosen, such that
\beq\skpb{\Lambda\EH}{\eh}_{\lzqom\times\lzqpeom}=0\mylabel{ehsenkrecht}\eeq
holds for all $\eh\in\calN(\Max,\Lambda,\omega)$\,. By this condition $\EH$ is uniquely determined.
\item[\rm\bf (v)] The solution operator introduced in {\rm\bf (iv)},
which we will denote by $\loesom$ as well, maps
$\big(\Lzqsom\times\Lzqpesom\big)\cap\calN(\Max,\Lambda,\omega)^\bot$ to
$\big(\Ronqtom\times\Dqpetom\big)\cap\calN(\Max,\Lambda,\omega)^{\bot_\Lambda}$
continuously for all $s,-t>1/2$\,.
\end{itemize}
\end{theo}

Here we denote the orthogonality corresponding to the
$\skp{\Lambda\,\cdot\,}{\,\cdot\,}_{\lzqom\times\lzqpeom}$-scalar product by $\bot_\Lambda$ and we put $\bot:=\bot_{\id}$\,.
Moreover, using the same technique introduced by Eidus in \cite{eidussm} for the classical
Maxwell equations we get

\begin{cor}\mylabel{expdecayeigen}
Let $\tau>1$\,, $\omega\in\rzon$ and $\EH\in\calN(\Max,\Lambda,\omega)$\,.
If  additionally $(\eps,\mu)\in\cq{2}(\Xi)\times\cqpe{2}(\Xi)$ with bounded derivatives for some exterior domain
$\Xi\subset\Omega$\,, then
\begin{align*}
\exp(t\,r)\cdot\EH&\in\big(\Ronqom{}\cap\eps^\me\Dqom{}\big)\times\big(\Dqpeom{}\cap\mu^\me\Ronqpeom{}\big)\quad,\\
\exp(t\,r)\cdot\EH&\in\qH{2}{q}{}{}(\tilde{\Xi})\times\qH{2}{q+1}{}{}(\tilde{\Xi})
\end{align*}
hold for all $t\in\rz$ and for all exterior domains $\tilde{\Xi}\subset\Xi$ with $\dist(\tilde{\Xi},\p\Xi)>0$\,,
i.e. eigensolutions decay exponentially.
\end{cor}

\begin{rem}\mylabel{polynomialdecayrem}
The polynomial resp. exponential decay of eigensolutions holds for arbitrary exterior domains $\Omega$\,,
i.e. $\Omega$ does not need to have the {\sf MLCP}.
\end{rem}

\begin{rem}\mylabel{bemfredholmRellich}
If the media are homogeneous and isotropic in the outside of some ball,
i.e. $\supp\Lambdad\cup(\rN\ohne\Omega)\subset U_{\rho}$ for some $\rho>0$\,,
then
$$\supp\EH\subset\ol{\Omega\cap U_{\rho}}$$
for all $\omega\in\rzon$ and $\EH\in\calN(\Max,\Lambda,\omega)$\,, since in this case $\EH$ solves Helmholtz' equation
$$(\Delta+\omega^2)\EH=(0,0)$$
in $A_{\rho}$ and therefore by Rellich's estimate \big(\cite{rellich} or \cite[p. 59]{leisbuch}\big) must vanish in $A_{\rho}$\,.
If the principle of unique continuation holds for our Maxwell system, then
$$\calN(\Max,\Lambda,\omega)=\big\{(0,0)\big\}\qquad.$$
\end{rem}

Moreover, using the a priori estimate of the limiting absorption principle and some indirect arguments
followed by the (trivial) decomposition of $\Lzqsom$ from \paulydecotrivdeco\,
we are able to prove stronger estimates for the solution operator $\loesom$ as the ones given in Theorem \ref{fredholm} (v).

\begin{cor}\mylabel{loesomabsch}
Let $\tau>1$\,, $s,-t>1/2$ and $K\Subset\cz_+\ohne\{0\}$ with $\ol{K}\cap\pP=\emptyset$ as well as $\Omega$ have the {\sf MLCP}. Then
\begin{itemize}
\item[\bf(i)] there exist constants $c>0$ and $\hat{t}>-1/2$\,, such that the estimate
$$\normb{\loesom\FG}_{\Rqtom\times\Dqpetom}+\normb{(r^\me S+\id)\loesom\FG}_{\Lzqom{\hat{t}}\times\Lzqpeom{\hat{t}}}$$
$$\leq c\cdot\normb{\FG}_{\Lzqsom\times\Lzqpesom}$$
holds true for all $\omega\in\ol{K}$ and $\FG\in\Lzqsom\times\Lzqpesom$\,. Especially the operator
$$\loesom\,:\,\Lzqsom\times\Lzqpesom\,\To\,\Ronqtom\times\Dqpetom$$
is equicontinuous w.~r.~t. $\omega\in\ol{K}$\,;
\item[\bf(ii)] the mapping
$$\Abb{\loes}{\ol{K}}{B\big(\Lzqsom\times\Lzqpesom,\Ronqtom\times\Dqpetom\big)}{\omega}{\loesom}$$
is (uniformly) continuous. \big(Here we denote the bounded linear operators from some normed space $X$
to some normed space $Y$ by $B(X,Y)$\,.\big)
\end{itemize}
\end{cor}

\section{The static problem}

To introduce our static solution concept we remind of the special forms $\bonqom$\,, $\bqpeom$ from \paulystaticsecgenstatic\,
and the `static Maxwell property' ({\sf SMP}), which guarantees their existence and also implies the {\sf MLCP}.
(If $\om$ is Lipschitz homeomorphic to a smooth exterior domain, then $\om$ possesses the {\sf SMP}.)
To work with these forms we may assume that $\om$ has the {\sf SMP}, and restrict our considerations to ranks $1\leq q\leq N$\,. \label{staticsection}

\begin{defini}\mylabel{defiloesstatik}
$\EH$ is a solution of $\Max(\Lambda,0,f,F,G,g,\zeta,\xi)$ with data
$$(f,F,G,g)\in\qLzom{q-1}{\loc}\times\qLzom{q}{\loc}\times\qLzom{q+1}{\loc}\times\qLzom{q+2}{\loc}$$
and $(\zeta,\xi)\in\cz^{d^q}\times\cz^{d^{q+1}}$\,, if and only if
\begin{align*}
\EH&\in\big(\Lzqom{>-\Nh}\cap\ronqlocom\cap\eps^\me\dqlocom\big)\\
&\qquad\qquad\qquad\times\big(\Lzqpeom{>-\Nh}\cap\mu^\me\ronqpelocom\cap\dqpelocom\big)
\end{align*}
solves the electro-magneto static system
\begin{align*}
\rot E&=G&&,&\pdiv\eps E&=f&&,&\skp{\eps E}{\bon{q}_\ell}_{\lzqom}&=\zeta_\ell&&,&\ell&=1,\dots,d^q&&,\\
\pdiv H&=F&&,&\rot\mu H&=g&&,&\skp{\mu H}{b^{q+1}_k}_{\lzqpeom}&=\xi_k&&,&k&=1,\dots,d^{q+1}&&.
\end{align*}
\end{defini}

Now we want to use \paulystaticgenstatictheo\, in the special case $s=0$ to solve the static problem
$\Max(\Lambda,0,f,F,G,g,\zeta,\xi)$\,. For this let $\eps$\,, $\mu$ be $\tau$-$\pc{1}$-admissible with $\tau>0$ as well as
$$\bDqsom:=\dqsnom\cap\bonqom^\bot\qqtext{,}\bRqsom:=\ronqsnom\cap\bqom^\bot\qquad,$$
where the latter is defined for $q\neq1$\,, and for $q\neq0$
$$\bWqsom=\bDom{q-1}{s}{0}\times\bRom{q+1}{s}{0}\times\cz^{d^q}\qquad.$$
Furthermore, for $s=0$ we put as usual $\bDom{q}{}{0}:=\bDom{q}{0}{0}$\,, $\bRom{q}{}{0}:=\bRom{q}{0}{0}$ and $\bWqom:=\bWom{q}{0}$\,.

\begin{theo}\mylabel{satzloesstatik}
For every data $(f,G,\zeta)\in\bWqom$ and $(F,g,\xi)\in\bWqpeom$ there exists a unique solution
$$\EH\in\big(\ronqom{-1}\cap\eps^\me\dqom{-1}\big)\times\big(\dqpeom{-1}\cap\mu^\me\ronqpeom{-1}\big)$$
of the electro-magneto static problem $\Max(\Lambda,0,f,F,G,g,\zeta,\xi)$ and the corresponding solution operator is continuous.
\end{theo}

\begin{rem}\mylabel{satzloesstatikrem}
For special data $(0,G,0)\in\bWqom$\,, $(F,0,0)\in\bWqpeom$\,, i.e.
$$\FG\in\bDom{q}{}{0}\times\bRom{q+1}{}{0}\qquad,$$
we will denote the corresponding continuous solution operator by
$$\loesn:\bDom{q}{}{0}\times\bRom{q+1}{}{0}\to\big(\ronqom{-1}\times\dqpeom{-1}\big)\cap\Lambda^\me\big(\bDom{q}{-1}{0}\times\bRom{q+1}{-1}{0}\big)\quad.$$
We note that $\loesn$ even maps $\bDqsom\times\bRqpesom$ to
$$\big(\ronqom{s-1}\times\dqpeom{s-1}\big)\cap\Lambda^\me\big(\bDom{q}{s-1}{0}\times\bRom{q+1}{s-1}{0}\big)$$
continuously for all $1-N/2<s<N/2$\,.
\end{rem}

\section{Low frequency asymptotics}

To approach the low frequency asymptotics of $\loesom$ we first have to be sure that $\pP$ does not accumulate at zero.
To this end first of all we derive a representation formula for the solutions of the homogeneous, isotropic whole space problem, i.e.
$\Omega:=\rN$ and $\Lambda:=\id$\,,
with the help of the fundamental solution $\Phi_{\omega,\nu}$
of the scalar Helmholtz operator in $\rN$ \label{asymsection}
$$\Delta+\omega^2\qqtext{,}\omega\in\cz_+\ohne\{0\}\qquad.$$
This one can be written as
$$\Phi_{\omega,\nu}(x)=\varphi_{\omega,\nu}\big(|x|\big)\qqtext{with}\varphi_{\omega,\nu}(t)=c_N\omega^\nu t^{-\nu}H_\nu^1(\omega t)\qquad,$$
where the constant $c_N$ only depends on the dimension $N$ and $H_\nu^1(z)$
represents Hankel's function of first kind for the index
$\nu:=(N-2)/2$\,. From now on we may additionally assume $N$ to be odd, since then
by the properties of Hankel's function \big(see e.g. \cite{magnus} or \cite[p. 76]{leisbuch}\big)
$\varphi_{\omega,\nu}$ and its first derivative can be estimated by
\beq\big|\varphi_{\omega,\nu}(t)\big|\leq c\,(t^{2-N}+t^\frac{1-N}{2})\qqtext{,}
\big|\varphi'_{\omega,\nu}(t)\big|\leq c\,(t^{1-N}+t^\frac{1-N}{2})\mylabel{grundloesungphiabsch}\eeq
uniformly in $t\in\rzp$ and $\omega\in K\Subset\cz_+$ with some constant $c>0$ depending only on $N$ and $K$\,.

From Remark \ref{bemfredholmRellich} we have (in the case $\Omega=\rN$)
$$\calN(\Max,\id,\omega)=\big\{(0,0)\big\}\qquad.$$
Thus $L_\omega$ is well defined on the whole of $\Lzqgeh\times\Lzqpegeh$\,,
if we denote $\loesom$ in the special case $\Omega=\rN$ and $\Lambda=\id$ by $L_\omega$\,.
Let $\omega\in\cz_+\ohne\{0\}$ and $\FG\in\cqun\times\cqpeun$\,.
Looking at $\EH:=L_\omega\FG$ we get
$$\EH\in(\qH{2}{q}{<\meh}{}\cap\cqu)\times(\qH{2}{q+1}{<\meh}{}\cap\cqpeu)$$
by regularity, e.g. \paulydissregsatzausseneins.
Applying $(M-\ie\omega)$ to $(M+\ie\omega)\EH=\FG$ and using $\ie\omega(\pdiv E,\rot H)=(\pdiv F,\rot G)$ we observe,
that $\EH$ satisfies
\beq(\Delta+\omega^2)\EH=\big(M-\ie\omega-\frac{\ie}{\omega}\Box\big)\FG=:\fg\in\cqun\times\cqpeun\mylabel{transformFGfg}\eeq
with $\Box:=\Delta-M^2=\zmat{\rot\pdiv}{0}{0}{\pdiv\rot}$\,.
We obtain $\EH=\eh$\,, where $\eh$ is the unique radiating solution of the whole space problem
\begin{align*}
(\Delta+\omega^2)\eh&=\fg\qquad,\\
\eh&\in\,\qH{2}{q}{<\meh}{}\times\qH{2}{q+1}{<\meh}{}\qquad,\\
\exp(-\ie\omega r)\cdot\eh&\in\,\qh{1}{q}{>-\frac{3}{2}}{}\times\qh{1}{q+1}{>-\frac{3}{2}}{}\qquad.
\end{align*}
For nonreal frequencies $\omega\in\cz_+\ohne\rz$ this is trivial, because again \paulydissregsatzausseneins\,
yields $\EH\in\qH{2}{q}{}{}\times\qH{2}{q+1}{}{}$\,. But then $\EH=\eh$ holds for real frequencies $\omega\in\rzon$ as well,
since one receives the solutions of both radiating problems with the principle of limiting absorption.

Using the representation formula for the solutions of the scalar Helmholtz equation,
which, for instance, can be found in
\cite[pp. 78/79, Remark 4.28]{leisbuch}, we can represent the Euclidean components of our forms
$E=E_I\pd x^I$ and $H=H_J\pd x^J$ by
$$E_I=f_I\star\Phi_{\omega,\nu}\qqtext{,}H_J=g_J\star\Phi_{\omega,\nu}\qquad.$$
Here we denote the scalar convolution in $\rN$ by $\star$\,.
For suitable $q$-forms $e=e_I\pd x^I$ and $h=h_I\pd x^I$ (Euclidean coordinates) we define the convolution
$$e\star h:=e_I\star h_I$$
simply as the sum of the componentwise scalar convolutions.
Furthermore, we have for suitable forms the rule of partial integration
\beq\rot e\star h(x)=e\star\pdiv h(x)\qquad.\mylabel{partInt}\eeq
With the special forms
$$\Phi_{\omega,\nu}^I:=\Phi_{\omega,\nu}\cdot\pd x^I$$
we get the representations
$$E_I=f\star\Phi_{\omega,\nu}^I\qqtext{,}H_J=g\star\Phi_{\omega,\nu}^J\qquad,$$
i.e. reminding of \eqref{transformFGfg}
\begin{align}
E_I&=\big(\pdiv G-\ie\omega F-\frac{\ie}{\omega}\rot\pdiv F\big)\star\Phi_{\omega,\nu}^I\qquad,\mylabel{DarstellungEHeins}\\
H_J&=\big(\rot F-\ie\omega G-\frac{\ie}{\omega}\pdiv\rot G\big)\star\Phi_{\omega,\nu}^J\qquad.\mylabel{DarstellungEHzwei}
\end{align}
Our next goal is to use the partial integration formula \eqref{partInt} to remove the second derivatives from $F$ and $G$\,.
Let us look at
$$(\pdiv G)\star\Phi_{\omega,\nu}^I$$
for example. Because of the compact support of $\FG$ we do not have to pay attention to the integrability of
$\Phi_{\omega,\nu}$ at infinity. By \eqref{grundloesungphiabsch} we can estimate $\Phi_{\omega,\nu}$ and $\nabla\Phi_{\omega,\nu}$
in $U_1$ by $|\Phi_{\omega,\nu}|\leq c\cdot r^{2-N}$\,, $|\nabla\Phi_{\omega,\nu}|\leq c\cdot r^{1-N}$ and thus we have
$\Phi_{\omega,\nu}\,,\,\nabla\Phi_{\omega,\nu}\in\text{\rm L}^{1}(U_1)$\,. With the cut-off functions
$$\psi_n(y):=\mbox{\boldmath$\eta$}\big(n\cdot|x-y|\big)\qqtext{,}n\in\nz\qquad,$$
which satisfy $\big|\nabla\psi_n(y)\big|\leq c\cdot|x-y|^\me$ uniformly in $n$\,, we have
$$\psi_n\cdot\vartheta_x\Phi_{\omega,\nu}\,,\,\nabla\psi_n\cdot\vartheta_x\Phi_{\omega,\nu}\,,\,\psi_n\cdot\nabla(\vartheta_x\Phi_{\omega,\nu})\in\text{\rm L}^{1}\big(U_1(x)\big)\qquad.$$
Therefore \eqref{partInt} yields
$$(\pdiv G_n)\star\Phi_{\omega,\nu}^I=G_n\star\rot\Phi_{\omega,\nu}^I$$
with $G_n:=\psi_n\cdot G$ and we obtain
$$(\pdiv G)\star\Phi_{\omega,\nu}^I=G\star\rot\Phi_{\omega,\nu}^I$$
by passing to the limit $n\to\infty$ and using Lebesgue's' dominated convergence theorem.
Using these partial integrations in \eqref{DarstellungEHeins} and \eqref{DarstellungEHzwei} we finally get the representations
\begin{align}
E_I&=G\star(\rot\Phi_{\omega,\nu}^I)-\ie\omega F\star\Phi_{\omega,\nu}^I-\frac{\ie}{\omega}(\pdiv F)\star(\pdiv\Phi_{\omega,\nu}^I)\qquad,\mylabel{DarstellungEHdrei}\\
H_J&=F\star(\pdiv\Phi_{\omega,\nu}^J)-\ie\omega G\star\Phi_{\omega,\nu}^J-\frac{\ie}{\omega}(\rot G)\star(\rot\Phi_{\omega,\nu}^J)\mylabel{DarstellungEHvier}
\end{align}
for any $\FG\in\cqun\times\cqpeun$ and $\EH=L_\omega\FG$\,.

\begin{theo}\mylabel{Darstellungsformel}
Let $0\neq\omega\in K\Subset\cz_+$ and $s\in(1/2,N/2)$\,, $t:=s-(N+1)/2$ as well as
$\FG\in\Dqs\times\Rqpes$\,. Then for $\EH:=L_\omega\FG$ the representation formulas
\begin{align*}
E&=\big(G\star(\rot\Phi_{\omega,\nu}^I)-\ie\omega F\star\Phi_{\omega,\nu}^I-\frac{\ie}{\omega}(\pdiv F)\star(\pdiv\Phi_{\omega,\nu}^I)\big)\cdot\pd x^I\qquad,\\
H&=\big(F\star(\pdiv\Phi_{\omega,\nu}^J)-\ie\omega G\star\Phi_{\omega,\nu}^J-\frac{\ie}{\omega}(\rot G)\star(\rot\Phi_{\omega,\nu}^J)\big)\cdot\pd x^J
\end{align*}
hold in the sense of $\Lzqt$ resp. $\Lzqpet$\,.
Furthermore, there exists a constant $c>0$\,, such that
\begin{align*}
\normb{L_\omega\FG}_{\Rqt\times\Dqpet}&\leq c\cdot\Big(\normb{\FG}_{\Lzqs\times\Lzqpes}\\
&\qquad\qquad+\frac{1}{|\omega|}\cdot\normb{(\pdiv F,\rot G)}_{\qLz{q-1}{s}\times\qLz{q+2}{s}}\Big)
\end{align*}
for all $0\neq\omega\in K$ and $\FG\in\Dqs\times\Rqpes$\,.
\end{theo}

\begin{proof}
We choose a sequence $\big(\FGn\big)_{n\in\nz}\subset\cqun\times\cqpeun$ converging to $\FG$ in $\Dqs\times\Rqpes$ as $n\to\infty$\,.
Theorem \ref{fredholm} (v) yields the convergence of $\EHn:=L_\omega\FGn$ to $\EH\in\Rqt\times\Dqpet$
in $\Rqt\times\Dqpet$ since $t<-1/2$\,.
By \eqref{DarstellungEHdrei} and \eqref{DarstellungEHvier} we may represent the forms $\EHn$ and observe that the involved
convolution kernels essentially consist of $\varphi_{\omega,\nu}\circ r$ and $\varphi_{\omega,\nu}'\circ r$\,.
Using \eqref{grundloesungphiabsch} these functions can be estimated by
$$\big|\varphi_{\omega,\nu}(r)\big|\,,\,\big|\varphi_{\omega,\nu}'(r)\big|
\leq c\,(r^{2-N}+r^{1-N}+r^\frac{1-N}{2})\leq c\,(r^{1-N}+r^\frac{1-N}{2})$$
uniformly in $r$ and $\omega$\,. From McOwen \cite[Lemma 1]{mcowen} we obtain,
that integral operators with kernels like $|x-y|^{s-t-N}$ map $\Lzs$ continuously to $\Lzt$\,, if
$$-N/2<t<s<N/2\qquad.$$
As a direct consequence the right hand sides of \eqref{DarstellungEHdrei} and \eqref{DarstellungEHvier}
define continuous linear operators from $\Lzs$ to $\Lzt$\,. This proves the asserted representation formulas.
By the differential equation it is sufficient to estimate $\normb{L_\omega\FG}_{\Lzqt\times\Lzqpet}$\,.
The uniform boundedness of the convolution operators w.~r.~t. $0\neq\omega\in K$ in the representation formulas
yields the desired estimate, which completes the proof.
\end{proof}

For $\gamma\in\rzp$ we put
$$\cz_{+,\gamma}:=\setb{\omega\in\cz_+}{|\omega|\leq\gamma}\qquad.$$
From now on we assume that $\Omega$ possesses the {\sf SMP}, $q\neq0$ and $\eps$\,, $\mu$ are $\tau$-$\pc{1}$-admissible
with order of decay
$$\tau>(N+1)/2\qquad.$$
We note that here it would be sufficient to demand the asymptotics
$$\epsd,\mud,\p_n\epsd,\p_n\mud=\calO(r^{-\tau})\qqtext{as}r\to\infty\qqtext{,}n=1,\dots,N\qquad.$$

\begin{lem}\mylabel{NullkeinHPAbsch}
Let $s\in(1/2,N/2)$ and $t:=s-(N+1)/2$\,.
\begin{itemize}
\item[\bf(i)] $\pP$ does not accumulate at zero. In particular $\pP$ has no accumulation point and
there exists some $\tilde{\omega}>0$\,, such that $\pP\cap\cz_{+,\tilde{\omega}}=\emptyset$\,.
\item[\bf(ii)] $\loesom$ is well defined on the whole of $\Lzqgehom\times\Lzqpegehom$ for all $\omega\in\cz_{+,\tilde{\omega}}\ohne\{0\}$\,.
\item[\bf(iii)] There exist constants $c>0$ and $0<\hat{\omega}\leq\tilde{\omega}$\,,
such that the estimate
\begin{align*}
&\qquad\quad\normb{\loesom\FG}_{\Lzqtom\times\Lzqpetom}\\
&\leq c\cdot\Big(\normb{\FG}_{\Lzqsom\times\Lzqpesom}+|\omega|^\me\cdot\normb{(\pdiv F,\rot G)}_{\qLzom{q-1}{s}\times\qLzom{q+2}{s}}\\
&\qquad\qquad+|\omega|^\me\cdot\sum_{\ell=1}^{d^q}\big|\skp{F}{\bon{q}_\ell}_{\lzqom}\big|+|\omega|^\me\cdot\sum_{\ell=1}^{d^{q+1}}\big|\skp{G}{b^{q+1}_\ell}_{\lzqpeom}\big|\Big)
\end{align*}
holds true for all $\omega\in\czpomd\ohne\{0\}$ and $\FG\in\Dqsom\times\Ronqpesom$\,.
\item[\bf(iv)] Especially for all $\FG\in\bDqsom\times\bRqpesom$ and $\omega\in\czpomd\ohne\{0\}$
$$\normb{\loesom\FG}_{\Lzqtom\times\Lzqpetom}\leq c\cdot\normb{\FG}_{\Lzqsom\times\Lzqpesom}\qquad.$$
\end{itemize}
The $\norm{\,\cdot\,}_{\Lzqtom\times\Lzqpetom}$-norm on the left hand sides of {\bf (iii)} and {\bf (iv)} may be replaced
by the natural norm in $\big(\Ronqtom\cap\eps^\me\Dqtom\big)\times\big(\mu^\me\Ronqpetom\cap\Dqpetom\big)$\,.
\end{lem}

\begin{proof}
First we prove the following:

For all $\check{\omega}>0$\,, $s\in(1/2,N/2)$ and $t:=s-(N+1)/2$ there exist constants
$c,\varrho>0$\,, such that the estimate
\begin{align}\begin{split}
&\qquad\qquad\normb{\EH}_{\Lzqtom\times\Lzqpetom}\\
&\leq c\cdot\Big(\normb{\FG}_{\Lzqsom\times\Lzqpesom}+\normb{\EH}_{\lzq(\Omega\cap U_{\varrho})\times\lzqpe(\Omega\cap U_{\varrho})}\\
&\qquad\qquad\qquad+|\omega|^\me\cdot\normb{(\pdiv F,\rot G)}_{\qLz{q-1}{s}(A_{r_0})\times\qLz{q+2}{s}(A_{r_0})}\Big)
\end{split}\mylabel{estimate}\end{align}
holds for all $\omega\in\cz_{+,\check{\omega}}\ohne\{0\}$\,, all
$$\FG\in\big(\Lzqsom\cap\Dqs(A_{r_0})\big)\times\big(\Lzqpesom\cap\Rqpes(A_{r_0})\big)$$
and all solutions $\EH$ of $\Max(\Lambda,\omega,F,G)$\,.

Let $\EH$ be a solution of $\Max(\Lambda,\omega,F,G)$ and $\EHs$ the extension by zero of $\eta\EH$ to $\rN$\,.
This one satisfies the radiation condition, is an element of $\Rqkmeh\times\Dqpekmeh$\,,
even of $\qH{1}{q}{<-\peh}{}\times\qH{1}{q+1}{<-\peh}{}$ by \paulydissregsatzausseneins, and solves
$$(M+\ie\omega)\EHs=\eta\FG+C_{M,\eta}\EH-\ie\omega\Lambdad\EHs=:\FGs\in\Dqs\times\Rqpes$$
in $\rN$ since $\tau>(N+1)/2>s+1/2$\,. Thus we obtain $\EHs=L_\omega\FGs$ and Theorem \ref{Darstellungsformel} yields a constant $c>0$
independent of $\omega$\,, $\FGs$ or $\EHs$ with
\begin{align}\begin{split}
&\qquad\qquad\normb{\EHs}_{\Lzqt\times\Lzqpet}\\
&\leq c\cdot\Big(\normb{\FGs}_{\Lzqs\times\Lzqpes}
+|\omega|^\me\cdot\normb{(\pdiv\tilde{F},\rot\tilde{G})}_{\qLz{q-1}{s}\times\qLz{q+2}{s}}\Big)\qquad.
\end{split}\mylabel{abschEHsganzraum}\end{align}
Furthermore, by the differential equations we get
\begin{align}
\ie\omega\pdiv\eps E&=\pdiv F&&,&\ie\omega\rot\mu H&=\rot G\mylabel{divrotGlEH}
\intertext{in $A_{r_0}$ and}
\ie\omega\pdiv\tilde{E}&=\pdiv\tilde{F}&&,&\ie\omega\rot\tilde{H}&=\rot\tilde{G}\mylabel{divrotGlEHs}
\end{align}
in $\rN$\,. Combining \eqref{abschEHsganzraum} and \eqref{divrotGlEHs} we have
\begin{align}\begin{split}
&\qquad\qquad\normb{\EH}_{\Lzqtom\times\Lzqpetom}\\
&\leq c\cdot\Big(\normb{\FG}_{\Lzqsom\times\Lzqpesom}+\normb{\EH}_{\lzq(\Omega\cap U_{r_2})\times\lzqpe(\Omega\cap U_{r_2})}\\
&\qquad\qquad\qquad+\normb{\EH}_{\Lzqom{s-\tau}\times\Lzqpeom{s-\tau}}+\normb{(\pdiv\tilde{E},\rot\tilde{H})}_{\qLz{q-1}{s}\times\qLz{q+2}{s}}\Big)
\end{split}\mylabel{estimatetwo}\end{align}
and using \eqref{divrotGlEH} we estimate the last term on the right hand side by
\begin{align*}
&\qquad\qquad\normb{(\pdiv\tilde{E},\rot\tilde{H})}_{\qLz{q-1}{s}\times\qLz{q+2}{s}}\\
&\leq c\cdot\Big(\normb{\EH}_{\lzq(\Omega\cap U_{r_2})\times\lzqpe(\Omega\cap U_{r_2})}\\
&\qquad\qquad\qquad\qquad+\normb{(\pdiv E,\rot H)}_{\qLz{q-1}{s}(\supp\eta)\times\qLz{q+2}{s}(\supp\eta)}\Big)\\
&\leq c\cdot\Big(\normb{\EH}_{\lzq(\Omega\cap U_{r_2})\times\lzqpe(\Omega\cap U_{r_2})}\\
&\qquad\qquad\qquad\qquad+\normb{(\pdiv\epsd E,\rot\mud H)}_{\qLz{q-1}{s}(\supp\eta)\times\qLz{q+2}{s}(\supp\eta)}\\
&\qquad\qquad\qquad\qquad+|\omega|^\me\cdot\normb{(\pdiv F,\rot G)}_{\qLz{q-1}{s}(A_{r_0})\times\qLz{q+2}{s}(A_{r_0})}\Big)\\
&\leq c\cdot\Big(\normb{\EH}_{\lzq(\Omega\cap U_{r_2})\times\lzqpe(\Omega\cap U_{r_2})}\\
&\qquad\qquad\qquad\qquad+\normb{\EH}_{\qH{1}{q}{s-\tau}{}(\supp\eta)\times\qH{1}{q+1}{s-\tau}{}(\supp\eta)}\\
&\qquad\qquad\qquad\qquad+|\omega|^\me\cdot\normb{(\pdiv F,\rot G)}_{\qLz{q-1}{s}(A_{r_0})\times\qLz{q+2}{s}(A_{r_0})}\Big)\qquad.
\end{align*}
Inserting this estimate into \eqref{estimatetwo}, using the regularity result \paulydissregkorausseneins,
the differential equation as well as \eqref{divrotGlEH} we finally get
\begin{align*}
&\qquad\qquad\normb{\EH}_{\Lzqtom\times\Lzqpetom}\\
&\leq c\cdot\Big(\normb{\FG}_{\Lzqsom\times\Lzqpesom}+\normb{\EH}_{\Lzqom{s-\tau}\times\Lzqpeom{s-\tau}}\\
&\qquad\qquad\qquad\qquad+|\omega|^\me\cdot\normb{(\pdiv F,\rot G)}_{\qLz{q-1}{s}(A_{r_0})\times\qLz{q+2}{s}(A_{r_0})}\Big)\qquad.
\end{align*}
By $\tau>(N+1)/2$ we have $s-\tau<t$ and thus \eqref{estimate} follows.

If we now assume that $0$ is an accumulation point of $\pP$ or the estimate in {\bf (iii)} is false, then there would exist
a sequence $(\omega_n)_{n\in\nz}\subset\cz_+\ohne\{0\}$ tending to zero and a data sequence
$$\big(\FGn\big)_{n\in\nz}\subset\big(\Dqsom\times\Ronqpesom\big)\cap\calN(\Max,\Lambda,\omega_n)^\bot$$
as well as a sequence of normed solutions $\EHn$ to $(M+\ie\omega_n\Lambda)\EHn=\FGn$ with
$\normb{\EHn}_{\Lzqtom\times\Lzqpetom}=1$ and
\begin{align*}
\normb{\FGn}_{\Lzqsom\times\Lzqpesom}&\xrightarrow{n\to\infty}0&&,\\
|\omega_n|^\me\cdot\normb{(\pdiv F_n,\rot G_n)}_{\qLzom{q-1}{s}\times\qLzom{q+2}{s}}&\xrightarrow{n\to\infty}0&&,\\
|\omega_n|^\me\cdot\big|\skp{F_n}{\bon{q}_\ell}_{\lzqom}\big|&\xrightarrow{n\to\infty}0&&,&\ell&=1,\dots,d^q&&,\\
|\omega_n|^\me\cdot\big|\skp{G_n}{b^{q+1}_k}_{\lzqpeom}\big|&\xrightarrow{n\to\infty}0&&,&k&=1,\dots,d^{q+1}&&.
\end{align*}
\big(In the case of {\bf (iii)} we have of course $\EHn=\loes_{\omega_n}\FGn$\big).
By the differential equation we get $\ie\omega_n(\pdiv\eps E_n,\rot\mu H_n)=(\pdiv F_n,\rot G_n)$ and thus
\begin{align}\begin{split}
&\qquad\normb{M\EHn}_{\Lzqtom\times\Lzqpetom}\\
&+\normb{(\pdiv\eps E_n,\rot\mu H_n)}_{\qLzom{q-1}{s}\times\qLzom{q+2}{s}}\xrightarrow{n\to\infty}0\qquad.
\end{split}\mylabel{Munddivrotkonv}\end{align}
Consequently $\EHn$ is bounded in
$$\big(\Ronqtom\cap\eps^\me\Dqtom\big)\times\big(\mu^\me\Ronqpetom\cap\Dqpetom\big)$$
and thus the {\sf MLCP} yields a subsequence, which we also denote by $\big(\EHn\big)_{n\in\nz}$\,,
converging for every $\tilde{t}<t$ in $\Lzqom{\tilde{t}}\times\Lzqpeom{\tilde{t}}$\,.
Because of \eqref{Munddivrotkonv} this sequence even converges in
$$\big(\Ronqom{\tilde{t}}\cap\eps^\me\Dqom{\tilde{t}}\big)\times\big(\mu^\me\Ronqpeom{\tilde{t}}\cap\Dqpeom{\tilde{t}}\big)$$
to the Dirichlet forms, let us say
$$\EH\in\dhqepsom{\tilde{t}}\times\mu^\me\dH{q+1}{\tilde{t}}{\mu^\me}(\Omega)\qquad.$$
Since $t=s-(N+1)/2\in(-N/2,-1/2)$ we may assume w.~l.~o.~g. $\tilde{t}\geq-N/2$\,.
Therefore by \eqref{dirichletint} we obtain
$$\EH\in\dhqepsom{}\times\mu^\me\dH{q+1}{}{\mu^\me}(\Omega)\qquad.$$
For $\ell=1,\dots,d^q$ we compute
\begin{align*}
&\qquad|\omega_n|^\me\cdot\big|\skp{F_n}{\bon{q}_\ell}_{\lzqom}\big|\xrightarrow{n\to\infty}0\\
&=|\omega_n|^\me\cdot\big|\ub{\skp{\pdiv H_n}{\bon{q}_\ell}_{\lzqom}}_{=0}+\ie\omega_n\skp{\eps E_n}{\bon{q}_\ell}_{\lzqom}\big|\\
&=\big|\skp{\eps E_n}{\bon{q}_\ell}_{\lzqom}\big|\xrightarrow{n\to\infty}\big|\skp{\eps E}{\bon{q}_\ell}_{\lzqom}\big|\qquad,
\end{align*}
i.e. $E\in\bonqom^{\bot_\eps}$\,. Analogously we see $H\in\bqpeom^{\bot_\mu}$\,.
Thus $\EH$ must vanish and finally \eqref{estimate} yields constants $c,\varrho>0$ independent of $n$ with
\begin{align*}
1&=\normb{\EHn}_{\Lzqtom\times\Lzqpetom}\\
&\leq c\cdot\Big(\normb{\FGn}_{\Lzqsom\times\Lzqpesom}\\
&\qquad\qquad+|\omega_n|^\me\cdot\normb{(\pdiv F_n,\rot G_n)}_{\qLz{q-1}{s}(A_{r_0})\times\qLz{q+2}{s}(A_{r_0})}\\
&\qquad\qquad+\normb{\EHn}_{\lzq(\Omega\cap U_{\varrho})\times\lzqpe(\Omega\cap U_{\varrho})}\Big)\xrightarrow{n\to\infty}0\qquad,
\end{align*}
a contradiction.
\end{proof}

We are ready to prove our main result:

\begin{theo}\mylabel{Asym}
Let $s\in(1/2,N/2)$\,, $t:=s-(N+1)/2$ and $\hat{\omega}$ be from Lemma \ref{NullkeinHPAbsch}.
Furthermore, let
$(\omega_n)_{n\in\nz}\subset\czpomd\ohne\{0\}$ be a sequence tending to $0$ and
$$\big(\FGn\big)_{n\in\nz}\subset\Dqsom\times\Ronqpesom$$
be a data sequence, such that
\begin{align*}
\FGn&\xrightarrow{n\to\infty}\FG&&\text{in}\quad\Lzqsom\times\Lzqpesom&&,\\
-\ie\omega_n^\me(\pdiv F_n,\rot G_n)&\xrightarrow{n\to\infty}\fg&&\text{in}\quad\qLzom{q-1}{s}\times\qLzom{q+2}{s}&&,\\
-\ie\omega_n^\me\skp{F_n}{\bon{q}_\ell}_{\lzqom}&\xrightarrow{n\to\infty}\zeta_\ell&&\text{in}\quad\cz\qtext{,}\ell=1,\dots,d^q&&,\\
-\ie\omega_n^\me\skp{G_n}{b^{q+1}_k}_{\lzqpeom}&\xrightarrow{n\to\infty}\xi_k&&\text{in}\quad\cz\qtext{,}k=1,\dots,d^{q+1}
\end{align*}
hold. Then $\big(\EHn\big)_{n\in\nz}:=\big(\loes_{\omega_n}\FGn\big)_{n\in\nz}$ converges for all $\tilde{t}<t$ in
$$\big(\Ronqom{\tilde{t}}\cap\eps^\me\Dqom{\tilde{t}}\big)\times\big(\mu^\me\Ronqpeom{\tilde{t}}\cap\Dqpeom{\tilde{t}}\big)$$
to $\EH$\,, the unique solution of the static problem $\Max(\Lambda,0,f,F,G,g,\zeta,\xi)$\,.
\end{theo}

\begin{proof}
From Lemma \ref{NullkeinHPAbsch} we get the boundedness of $\big(\EHn\big)_{n\in\nz}$ in
$$\big(\Ronqtom\cap\eps^\me\Dqtom\big)\times\big(\mu^\me\Ronqpetom\cap\Dqpetom\big)\qquad.$$
Thus by the {\sf MLCP} we can extract a subsequence,
which we will denote by $\big(\EHn\big)_{n\in\nz}$ as well, such that
$$\EHn\xrightarrow{n\to\infty}:\EHs\qqtext{in}\Lzqom{\tilde{t}}\times\Lzqpeom{\tilde{t}}$$
holds for all $\tilde{t}\in(-N/2,t)$\,. The differential equation
$$M\EHn+\ie\omega_n\Lambda\EHn=\FGn$$
and the assumptions yield
\begin{align*}
M\EHn&\xrightarrow{n\to\infty}\FG&&\text{in}&\Lzqtom&\times\Lzqpetom\qquad,\\
(\pdiv\eps E_n,\rot\mu H_n)&\xrightarrow{n\to\infty}\fg&&\text{in}&\qLz{q-1}{s}(\Omega)&\times\qLz{q+2}{s}(\Omega)\qquad.
\end{align*}
For $k=1,\dots,d^{q+1}$ we compute
$$\skp{\mu H_n}{b^{q+1}_k}_{\lzqpeom}=\frac{\ie}{\omega_n}\ub{\skp{\rot E_n}{b^{q+1}_k}_{\lzqpeom}}_{=0}-\frac{\ie}{\omega_n}\skp{G_n}{b^{q+1}_k}_{\lzqpeom}\xrightarrow{n\to\infty}\xi_k$$
and analogously $\skp{\eps E_n}{\bon{q}_\ell}_{\lzqom}\xrightarrow{n\to\infty}\zeta_\ell$ for $\ell=1,\dots,d^q$\,.
Thus $\EHs$ is an element of
$$\big(\ronqom{>-\Nh}\cap\eps^\me\dqom{>-\Nh}\big)\times\big(\mu^\me\ronqpeom{>-\Nh}\cap\dqpeom{>-\Nh}\big)$$
solving the electro-magneto static system
\begin{align*}
\rot\tilde{E}&=G&&,&\pdiv\tilde{H}&=F&&,\\
\pdiv\eps\tilde{E}&=f&&,&\rot\mu\tilde{H}&=g&&,\\
\big[\skp{\eps\tilde{E}}{\bon{q}_\ell}_{\lzqom}\big]_{\ell=1}^{d^q}&=\zeta&&,&\big[\skp{\mu\tilde{H}}{b^{q+1}_k}_{\lzqpeom}\big]_{k=1}^{d^{q+1}}&=\xi&&.
\end{align*}
For the difference $\eh:=\EH-\EHs$ we obtain
$$\eh\in\big(\dhqepsom{>-\Nh}\cap\bonqom^{\bot_\eps}\big)\times\big(\mu^\me\dH{q+1}{>-\Nh}{\mu^\me}(\Omega)\cap\bqpeom^{\bot_\mu}\big)$$
and even $\eh\in\Lzqom{}\times\Lzqpeom{}$ again by \eqref{dirichletint}.
Thus $\eh$ must vanish and because of the uniqueness of the limit $\EHs=\EH$ even the whole sequence
$\big(\EHn\big)_{n\in\nz}$ must converge to $\EH$ in $\Lzqom{<t}\times\Lzqpeom{<t}$\,.
\end{proof}

\begin{cor}\mylabel{Asymcor}
Let $s$\,, $t$\,, $\hat{\omega}$ be as in Theorem \ref{Asym} and
$\FG\in\bDqsom\times\bRqpesom$\,.
Then the solutions $\loesom\FG$ of the time-harmonic problem $\Max(\Lambda,\omega,F,G)$ converge for all $\tilde{t}<t$
in $\Ronqom{\tilde{t}}\times\Dqpeom{\tilde{t}}$ to $\loesn\FG$\,,
the unique solution of the static problem $\Max(\Lambda,0,0,F,G,0,0,0)$\,, as $\omega\in\czpomd\ohne\{0\}$ tends to zero.
\end{cor}

By a similar indirect argument \big(see \paulydisskorsiebenfuenf\big) we obtain

\begin{cor}\mylabel{korstetiginNull}
Let $s\in(1/2,N/2)$\,, $t:=s-(N+1)/2$\,, $\hat{\omega}$ be from Lemma \ref{NullkeinHPAbsch} and
$B_{s,t}$ be the Banach space of bounded linear operators from the Hilbert spaces
$$\bDqsom\times\bRqpesom\qqtext{to}\Ronqtom\times\Dqpetom\qquad.$$
Then $\norm{\loesom}_{B_{s,t}}$ is uniformly bounded w.~r.~t. $\omega\in\czpomd$ (even for $\omega=0$\,!).
Moreover, the mapping
$$\Abb{\loes}{\czpomd}{B_{s,\tilde{t}}}{\omega}{\loesom}$$
is (uniformly) continuous for all $\tilde{t}<t$\,.
\end{cor}

\begin{rem}\mylabel{korstetiginNullrem}
Clearly $\Ronqtom\times\Dqpetom$ resp. $\Ronqom{\tilde{t}}\times\Dqpeom{\tilde{t}}$
may be replaced by its closed subspace
$$\big(\Ronqtom\times\Dqpetom\big)\cap\Lambda^\me\big(\bDom{q}{t}{0}\times\bRom{q+1}{t}{0}\big)$$
resp.
$$\big(\Ronqom{\tilde{t}}\times\Dqpeom{\tilde{t}}\big)\cap\Lambda^\me\big(\bDom{q}{\tilde{t}}{0}\times\bRom{q+1}{\tilde{t}}{0}\big)\qquad.$$
\end{rem}

\begin{cor}\mylabel{Asymwholespace}
Let $s$\,, $t$\,, $\hat{\omega}$\,, $(\omega_n)$ be as in Theorem \ref{Asym} as well as
$$\big(\FGn\big)\subset\Lzqsom\times\Lzqpesom\qquad,$$
which may be decomposed by \paulydecodecoknh, such that
\begin{align*}
\FGn&=\Lambda\FrGdn+\FdGrn
\intertext{with}
\FrGdn&\in\big(\bRqsom\dotplus\Lin\bonqom\big)\times\big(\bDom{q+1}{s}{0}\dotplus\Lin\bqpeom\big)\qquad,\\
\FdGrn&\in\bDqsom\times\bRqpesom\qquad.
\end{align*}
Moreover, let $\big(\FdGrn\big)$ converge to some $\FdGr$ in $\Lzqsom\times\Lzqpesom$
as well as $\big(-\frac{\ie}{\omega_n}\FrGdn\big)$
converge to some $\ErHd$ in $\Lzqom{\tilde{t}}\times\Lzqpeom{\tilde{t}}$ for all ${\tilde{t}}<t$\,.
Then $\big(\EHn\big):=\big(\loes_{\omega_n}\FGn\big)$ converges for all $\tilde{t}<t$ in
$\Lzqom{\tilde{t}}\times\Lzqpeom{\tilde{t}}$ to the form $\EH=\ErHd+\loesn\FdGr$\,.
\end{cor}

\begin{proof}
$\big(\loes_{\omega_n}\FdGrn\big)$ converges to $\loesn\FdGr$ by Corollary \ref{Asymcor}.
Moreover, of course
$$\loes_{\omega_n}\Lambda\FrGdn=-\frac{\ie}{\omega_n}\FrGdn$$
holds.
\end{proof}

\section{Inhomogeneous boundary data}

We want to finish this paper by discussing inhomogeneous boundary data. \label{inhomo}

Recently Weck showed in \cite{wecklip},
how one may obtain traces of differential forms on Lipschitz boundaries.
Let $\Xi$ be a bounded Lipschitz domain in $\rN$\,. Then we know from \cite[Theorem 3]{wecklip}
the existence of a linear and continuous tangential trace operator (using for a moment the notations from there)
$$\Abb{\mathcal{T}}{\pRq{}(\Xi)}{R^{-1/2,q}(\p\Xi)}{E}{\iota^*E}\qquad.$$
Moreover, he proved in \cite[Theorem 4]{wecklip} that $\mathcal{T}$ is surjective, i.e. the existence
of a corresponding linear and continuous tangential extension operator (a right inverse)
$$\map{\mathcal{T}^\me}{R^{-1/2,q}(\p\Xi)}{\pRq{}(\Xi)}\qquad.$$
Let $\eps$ be a $0$-admissible transformation. Applying the usual Helmholtz decomposition
$$\Lzq{}(\Xi)=\ol{\rot\pR{q-1}{}{}{\circ}(\Xi)}\oplus_\eps\dhqeps{}(\Xi)\oplus_\eps\eps^\me\ol{\pdiv\Dqpe{}(\Xi)}$$
we receive a linear and continuous tangential extension operator with range in
$$\pRq{}(\Xi)\cap\eps^\me\ol{\pdiv\Dqpe{}(\Xi)}\subset\pRq{}(\Xi)\cap\eps^\me\Dqn{}(\Xi)\qquad.$$
If we assume now that $\om$ possesses a Lipschitz boundary (This implies the {\sf SMP}.), then we get by an usual cut-off-technique
for any $s\in\rz$ a linear and continuous tangential trace operator
$$\map{\gt}{\rqsom}{\xxrq(\dom):=R^{-1/2,q}(\dom)}$$
and a corresponding linear and continuous tangential extension operator
$$\map{\chgt}{\xxrq(\dom)}{\rqvoxom\cap\eps^\me\dqvoxom\subset\rqsom\cap\eps^\me\dqsom}$$
satisfying $\gt\chgt=\id$ on $\xxrq(\dom)$\,. We note that the kernel of $\gt$ equals $\ronqsom$ and
that $\gt$ may be defined even on $\rqlocomq$\,. $\chgt$ may be chosen, such that
$\supp\chgt\lambda\subset\ol{\om\cap U_{r_2}}$ holds for all $\lambda\in\xxrq(\dom)$\,.

Let $\FG\in\Lzqlocom\times\Lzqpelocom$ and $\lambda\in\xxrq(\dom)$ be some boundary data.
We want to discuss the solvability of the time-harmonic Maxwell system
\begin{equation}(M+\ie\omega\Lambda)\EH=\FG\qqtext{,}\gt E=\lambda\mylabel{lambdaequation}\end{equation}
using the results obtained so far. By definition we have
$$\El:=\chgt\lambda\in\rqvoxom\cap\eps^\me\dqvoxom$$
and with the ansatz
\begin{equation}\EH:=\EHs+(\El,0)\mylabel{lambdaansatz}\end{equation}
the equations \eqref{lambdaequation} turn to
\begin{equation}(M+\ie\omega\Lambda)\EHs=\FGs\qqtext{,}\gt\tilde{E}=0\mylabel{lambdanullequation}\end{equation}
with $\FGs:=\FG-(\ie\omega\eps\El,\rot\El)$\,. Thus we are looking for $\tilde{E}\in\ronqlocomq$
and we can use the results from the previous sections.
Moreover, for any $s\in\rz$ we clearly have
$$\FG\in\Lzqsom\times\Lzqpesom\quad\Equi\quad\FGs\in\Lzqsom\times\Lzqpesom$$
and for nonreal frequencies $\omega\in\cz\ohne\rz$ and $\FG\in\Lzqom{}\times\Lzqpeom{}$
we easily get unique square integrable time-harmonic solutions
\begin{align*}
\EH:=&\loesom\FGs+(\El,0)\\
=&\loesom\FG-\loesom(\ie\omega\eps\chgt\lambda,\rot\chgt\lambda)+(\chgt\lambda,0)\in\Rqom{}\times\Dqpeom{}\quad.
\end{align*}
We denote the continuous solution operator by
$$\Abb{\Loesom}{\Lzqom{}\times\Lzqpeom{}\times\xxrq(\dom)}{\Rqom{}\times\Dqpeom{}}{(F,G,\lambda)}{\EH}$$
and note $\loesom=\Loesom(\,\cdot\,,\,\cdot\,,0)$\,.

To establish a solution theory for non vanishing real frequencies $\omega\in\rzon$ and data $\FG\in\Lzqgehom\times\Lzqpegehom$
with our Fredholm theory from Theorem \ref{fredholm} we consider $\tau$-admissible transformations $(\eps,\mu)$
with some $\tau>1$\,. Using the ansatz \eqref{lambdaansatz} we only have to guarantee
$$\FGs\quad\bot\quad\calN(\Max,\Lambda,\omega)\qquad.$$
Let $\eh\in\calN(\Max,\Lambda,\omega)$\,. We compute
\begin{align*}
&\qquad\skpb{\FGs}{\eh}_{\lzqom\times\lzqpeom}\\
&=\skpb{\FG}{\eh}_{\lzqom\times\lzqpeom}-\skp{\rot\El}{h}_{\lzqpeom}+\skp{\El}{\ie\omega\eps e}_{\lzqom}\\
&=\skpb{\FG}{\eh}_{\lzqom\times\lzqpeom}-\skp{\rot\El}{h}_{\lzqpeom}-\skp{\El}{\pdiv h}_{\lzqom}\\
&=\skpb{\FG}{\eh}_{\lzqom\times\lzqpeom}-\skp{T_t\El}{T_nh}_{\lzqom\times\lzqpeom}
\end{align*}
with $T_t\Phi:=(\Phi,\rot\Phi)$ and $T_n\Psi:=(\pdiv\Psi,\Psi)$\,.\\

\begin{rem}\mylabel{spuren}
Assuming more regularity of $\om$\,, i.e. $\om\in\pc{2}$\,, and $\mu$\,, i.e. $\mu\in\pc{1}$\,, by Stokes' theorem
$$\skp{T_t\El}{T_nh}_{\lzqom\times\lzqpeom}=\skp{\gt\El}{\xgn h}_{\qH{\meh}{q}{}{}(\dom)}
=\skp{\lambda}{\xgn h}_{\qH{\meh}{q}{}{}(\dom)}$$
holds, since then by regularity $h$ is an element of $\qHom{1}{q+1}{}{}$
and thus $\xgn h$ is an element of $\qH{1/2}{q}{}{}(\dom)$\,,
where $\xgn=\pm\cast\iota^**$ denotes the usual normal trace.
\big(Here $\cast$ denotes the star-operator on the submanifold $\dom$ of $\omq$ and
$\bds\skp{\,\cdot\,}{\,\cdot\,}_{\qH{\meh}{q}{}{}(\dom)}\eds$
the duality between $\qH{\meh}{q}{}{}(\dom)$ and $\qH{\frac{1}{2}}{q}{}{}(\dom)$\,.\big)
\end{rem}

These considerations yield the following solution concept for $\omega\in\rzon$\,:

We call $\EH$ a solution of the radiation problem $\Max(\Lambda,\omega,F,G,\lambda)$\,, if and only if
\begin{itemize}
\item\quad $\EH\in\Rqkmehom\times\Dqpekmehom$\quad,
\item\quad $(M+\ie\omega\Lambda)\EH=\FG\qtext{and}\gt E=\lambda$\quad,
\item\quad $(r^\me S+\id)\EH\in\Lzqgmeh(\Omega)\times\Lzqpegmeh(\Omega)$\quad.
\end{itemize}

\begin{theo}\mylabel{fredholmboundary}
Let $(\eps,\mu)$ be $\tau$-admissible with $\tau>1$\,.
For all $\omega\in\rzon$\,, $\lambda\in\xxrq(\dom)$ and $\FG\in\Lzqgehom\times\Lzqpegehom$
there exists a solution $\EH$ of
$\Max(\Lambda,\omega,F,G,\lambda)$\,, if and only if
$$\skpb{\FG}{\eh}_{\lzqom\times\lzqpeom}=\skp{T_t\chgt\lambda}{T_nh}_{\lzqom\times\lzqpeom}$$
for all $\eh\in\calN(\Max,\Lambda,\omega)$\,.
The solution can be chosen, such that
$$\EH\quad\bot_\Lambda\quad\calN(\Max,\Lambda,\omega)\qquad.$$
Then by this condition the solution $\EH$ is uniquely determined and the solution operator
$$\Abb{\Loesom}{\Lzqgehom\times\Lzqpegehom\times\xxrq(\dom)}{\Rqkmehom\times\Dqpekmehom}{(F,G,\lambda)}{\EH}\qquad,$$
where $\EH=\loesom\FG-\loesom(\ie\omega\eps\chgt\lambda,\rot\chgt\lambda)+(\chgt\lambda,0)$\,,
is continuous in the sense of Theorem \ref{fredholm} (v).
\end{theo}

Now we need an adequate static solution theory to describe the asymptotic behaviour of $\Loesom$\,.

We call $\EH$ a solution of $\Max(\Lambda,0,f,F,G,g,\zeta,\xi,\lambda,\varkappa)$\,, if and only if
$$\EH\in\big(\rqom{>-\Nh}\cap\eps^\me\dqom{>-\Nh}\big)\times\big(\mu^\me\rqpeom{>-\Nh}\cap\dqpeom{>-\Nh}\big)$$
and
\begin{align*}
\rot E&=G&&,&\pdiv H&=F&&,\\
\pdiv\eps E&=f&&,&\rot\mu H&=g&&,\\
\big[\skp{\eps E}{\bon{q}_\ell}_{\lzqom}\big]_{\ell=1}^{d^q}&=\zeta&&,&\big[\skp{\mu H}{b^{q+1}_k}_{\lzqpeom}\big]_{k=1}^{d^{q+1}}&=\xi&&,\\
\gt E&=\lambda&&,&\gt\mu H&=\varkappa&&
\end{align*}
hold.

For the rest of this paper let $q\neq0$\,.
From \paulystaticstatloesinhom\, (in the special case $s=0$) we get

\begin{theo}\mylabel{staticboundarysolutiontheo}
Let $(\eps,\mu)$ be $\tau$-$\pc{1}$-admissible with $\tau>0$\,.
Then for all $f\in\bDom{q-1}{}{0}$\,, $F\in\bDom{q}{}{0}$\,,
$\zeta\in\cz^{d^q}$\,, $\xi\in\cz^{d^{q+1}}$
and all $G\in\rqpenom{}$\,, $g\in\pr{q+2}{}{0}{}(\om)$\,,
$\lambda\in\xxrq(\dom)$\,, $\varkappa\in\xxrqpe(\dom)$ satisfying
\begin{align*}
\Rot\lambda&=\gt G&&\wedge&\bigwedge_{b\in\bqpeom}\skp{G}{b}_{\lzqpeom}&=\skp{T_t\chgt\lambda}{T_nb}_{\lzqom\times\lzqpeom}\\
&&&&&=\skp{\rot\chgt\lambda}{b}_{\lzqpeom}&&,\\
\Rot\varkappa&=\gt g&&\wedge&\bigwedge_{b\in\B^{q+2}(\om)}\skp{g}{b}_{\qlzom{q+2}}&=\skp{T_t\chgt\varkappa}{T_nb}_{\lzqpeom\times\qlzom{q+2}}\\
&&&&&=\skp{\rot\chgt\varkappa}{b}_{\qlzom{q+2}}
\end{align*}
there exists a unique solution
$$\EH\in\big(\rqom{-1}\cap\eps^\me\dqom{-1}\big)\times\big(\mu^\me\rqpeom{-1}\cap\dqpeom{-1}\big)$$
of $\Max(\Lambda,0,f,F,G,g,\zeta,\xi,\lambda,\varkappa)$\,. The solution depends continuously on the data.
\end{theo}

\begin{rem}\mylabel{spurenzwei}
Once again assuming more regularity of $\om$\,, i.e. $\om\in\pc{2}$\,, we have
\begin{align*}
\skp{T_t\chgt\lambda}{T_nb}_{\lzqom\times\lzqpeom}&=\skp{\lambda}{\xgn b}_{\qH{\meh}{q}{}{}(\dom)}
\intertext{resp.}
\skp{T_t\chgt\varkappa}{T_nb}_{\lzqpeom\times\qlzom{q+2}}&=\skp{\varkappa}{\xgn b}_{\qH{\meh}{q+1}{}{}(\dom)}\qquad.
\end{align*}
\end{rem}

Finally we are ready to prove our last result:

\begin{theo}\mylabel{asymboundaryone}
Let $(\eps,\mu)$ be $\tau$-$\pc{1}$-admissible with $\tau>(N+1)/2$\,.
Let $s\in(1/2,N/2)$ and $t:=s-(N+1)/2$ as well as $\hat{\omega}$
be from Lemma \ref{NullkeinHPAbsch}. Moreover, let
$(\omega_m)_{m\in\nz}\subset\czpomd\ohne\{0\}$ be a sequence tending to zero and
$$\big(\FGm\big)_{m\in\nz}\subset\Dqsom\times\Rqpesom\qqtext{,}(\lambda_m)_{m\in\nz}\subset\xxrq(\dom)$$
be some data sequences with
$$\gt G_m=\Rot\lambda_m\qquad,$$
such that
\begin{align*}
\lambda_m&\xrightarrow{m\to\infty}\lambda&&\text{in}\quad\xxrq(\dom)&&,\\
\FGm&\xrightarrow{m\to\infty}\FG&&\text{in}\quad\Lzqsom\times\Lzqpesom&&,\\
-\ie\omega_m^\me(\pdiv F_m,\rot G_m)&\xrightarrow{m\to\infty}\fg&&\text{in}\quad\qLzom{q-1}{s}\times\qLzom{q+2}{s}&&,\\
-\ie\omega_m^\me\skp{F_m}{\bon{q}_\ell}_{\lzqom}&\xrightarrow{m\to\infty}\zeta_\ell&&\text{in}\quad\cz\qtext{,}\ell=1,\dots,d^q&&,\\
-\ie\omega_m^\me\Big(\skp{G_m}{b^{q+1}_k}_{\lzqpeom}\quad&&&&&\\
-\skp{\rot\chgt\lambda_m}{b^{q+1}_k}_{\lzqpeom}\Big)&\xrightarrow{m\to\infty}\xi_k&&\text{in}\quad\cz\qtext{,}k=1,\dots,d^{q+1}
\end{align*}
hold. Then $\big(\EHm\big)_{m\in\nz}:=\big(\cS_{\omega_m}(F_m,G_m,\lambda_m)\big)_{m\in\nz}$ converges for all $\tilde{t}<t$ in
$$\big(\Rqom{\tilde{t}}\cap\eps^\me\Dqom{\tilde{t}}\big)\times\big(\mu^\me\Rqpeom{\tilde{t}}\cap\Dqpeom{\tilde{t}}\big)$$
to $\EH$\,, the unique solution of the static problem $\Max(\Lambda,0,f,F,G,g,\zeta,\xi,\lambda,0)$\,.
\end{theo}

\begin{proof}
From Theorem \ref{fredholmboundary} and \eqref{lambdaansatz} we have $\EHm=\EHms+(E_{\lambda_m},0)$ with $E_{\lambda_m}:=\chgt\lambda_m$\,,
$\EHms:=\loes_{\omega_m}\FGms$ and
$$\Fms:=F_m-\ie\omega_m\eps E_{\lambda_m}\qqtext{,}\Gms:=G_m-\rot E_{\lambda_m}\qquad.$$
Because of the compact support of $E_{\lambda_m}$ and the continuity of $\chgt$ we have
$$E_{\lambda_m}\xrightarrow{m\to\infty}E_\lambda:=\chgt\lambda\qqtext{in}\rqsom\cap\eps^\me\dqsom$$
for all $s\in\rz$\,. Moreover, $\FGms$ fulfills the assumptions of Theorem \ref{Asym}.
Thus $\EHms$ converges for all $\tilde{t}<t$ in
$$\big(\Ronqom{\tilde{t}}\cap\eps^\me\Dqom{\tilde{t}}\big)\times\big(\mu^\me\Ronqpeom{\tilde{t}}\cap\Dqpeom{\tilde{t}}\big)$$
to $\EHs$\,, the unique solution of $\Max(\Lambda,0,\tilde{f},\tilde{F},\tilde{G},\tilde{g},\tilde{\zeta},\tilde{\xi})$
with $\tilde{F}=F$\,, $\tilde{g}=g$\,, $\tilde{\xi}=\xi$ and
$$\tilde{G}=G-\rot E_\lambda\qtext{,}\tilde{f}=f-\pdiv\eps E_\lambda\qtext{,}\tilde{\zeta}=\zeta-\big[\skp{\eps E_\lambda}{\bon{q}_\ell}_{\lzqom}\big]_{\ell=1}^{d^q}\quad.$$
We obtain $\EHm\xrightarrow{m\to\infty}\EH:=\EHs+(E_\lambda,0)$ with the asserted mode of convergence
and clearly $\EH$ is the unique solution of the static problem
$$\Max(\Lambda,0,f,F,G,g,\zeta,\xi,\lambda,0)\qquad,$$
which completes the proof.
\end{proof}

\begin{acknow}
This research was supported by the {\it Deutsche Forschungsgemeinschaft}
via the project {\sf `We 2394: Untersuchungen der Spektralschar verallgemeinerter
Maxwell-Operatoren in unbeschr\"ankten Gebieten'}.

The author is particularly indebted to his academic teachers Norbert Weck and Karl-Josef Witsch 
for introducing him to the field.
\end{acknow}

\end{document}